\documentclass[12pt]{article}

\usepackage{latexsym,amsmath,amscd,amssymb,amsthm,graphics}
\usepackage{enumerate,framed,comment}
\usepackage{color}
\usepackage{verbatim}
\usepackage[margin=2.5cm]{geometry}
\usepackage{epsfig,epstopdf,color,enumerate}
\usepackage{caption}
\usepackage{graphicx, subfigure, psfrag} 

\usepackage{placeins}
\usepackage{sidecap}

\usepackage[colorlinks]{hyperref}
\usepackage{url}  %
\usepackage[percent]{overpic}

\usepackage[T1]{fontenc}
\usepackage[colorlinks]{hyperref}

\usepackage[all]{xy}

\newcommand{\rem}[1]{}
\makeatletter

\@addtoreset{figure}{section}
\def\thefigure{\thesection.\@arabic\c@figure}
\def\fps@figure{h, t}
\@addtoreset{table}{bsection}
\def\thetable{\thesection.\@arabic\c@table}
\def\fps@table{h, t}
\@addtoreset{equation}{section}

\makeatother

\rem{
\textwidth 6.5 truein
\oddsidemargin 0 truein
\evensidemargin .2 truein
\topmargin -.6 truein
\textheight 9.1 in
}

{

\begin{document}

\newtheorem{theorem}{Theorem}[section]
\newtheorem{definition}[theorem]{Definition}
\newtheorem{lemma}[theorem]{Lemma}
\newtheorem{remark}[theorem]{Remark}
\newtheorem{proposition}[theorem]{Proposition}
\newtheorem{corollary}[theorem]{Corollary}
\newtheorem{example}[theorem]{Example}

\def\below#1#2{\mathrel{\mathop{#1}\limits_{#2}}}



\title{Invariant higher-order variational problems II}
\author{
Fran\c{c}ois Gay-Balmaz$^{1}$, Darryl D. Holm$^{2}$, David M. Meier$^{2}$, 
\\Tudor S. Ratiu$^{3}$, Fran\c{c}ois-Xavier Vialard$^{4}$
}
\addtocounter{footnote}{1}
\footnotetext{Laboratoire de M\'et\'eorologie Dynamique, \'Ecole Normale Sup\'erieure/CNRS, Paris, France. 
\texttt{gaybalma@lmd.ens.fr}
\addtocounter{footnote}{1} }
\footnotetext{Department of Mathematics, Imperial College, London SW7 2AZ, UK. 
\texttt{d.holm@ic.ac.uk, d.meier09@ic.ac.uk}
\addtocounter{footnote}{1}}
\footnotetext{Section de
Math\'ematiques and Bernoulli Center, \'Ecole Polytechnique F\'ed\'erale de
Lausanne,
CH--1015 Lausanne, Switzerland.
\texttt{tudor.ratiu@epfl.ch}
\addtocounter{footnote}{1}}
\footnotetext{Centre de Recherche en Math\'ematiques de la D\'ecision, Universit\'e Paris-Dauphine, Paris, France.
\texttt{fxvialard@normalesup.org}
\addtocounter{footnote}{1}}
\date{{\it Fondly remembering our late friend Jerry Marsden}\\
}
\maketitle

\makeatother
\maketitle




\begin{abstract} 
 Motivated by applications in computational anatomy, we consider a second-order problem in the calculus of variations on object manifolds that are acted upon by Lie groups of smooth invertible transformations. This problem leads to solution curves known as Riemannian cubics on object manifolds that are endowed with normal metrics. The prime examples of such object manifolds are the symmetric spaces. We characterize the class of cubics on object manifolds that can be lifted horizontally to cubics on the group of transformations. Conversely, we show that certain types of non-horizontal geodesics on the group of transformations project to cubics. Finally, we apply second-order Lagrange--Poincar\'e reduction to the problem of Riemannian cubics on the group of transformations. This leads to a reduced form of the equations that reveals the obstruction for the projection of a cubic on a transformation group to again be a cubic on its object manifold.
\end{abstract}
\newpage
\tableofcontents


\section{Introduction}\label{Intro-sec}

In this section, we summarize the main content of the paper, motivated by potential applications of first-order and higher-order trajectory planning problems in computational anatomy.
 
\subsection{General Background}
\paragraph{Geodesic matching in computational anatomy.} The new science of computational anatomy (CA) is concerned with quantitative  comparisons of shape, in particular the shapes of organs in the human body \cite{GrMi1998}. In CA, shapes are defined by spatial distributions of various types of geometric data structures, such as points (landmarks), spatially embedded curves or surfaces (boundaries), intensity, or density (regions), or tensors that encode  local  orientation of muscle fibers, etc.
 A fruitful approach in this burgeoning field applies the large deformation matching (LDM) method. In the LDM method, shapes are compared by measuring the relative deformation required to match one shape to another under  the action of the  diffeomorphism group \cite{DuGrMi1998, Tr1998}. This approach follows D'Arcy Thompson's inspired proposal to study comparisons of shapes by transforming one shape into another \cite{DThompson1942}. More specifically, it follows the Grenander deformable template paradigm \cite{Gr1993}. 

If a Lie group $G$ is equipped with an invariant metric, then its action on a smooth manifold $Q$ induces a metric on $Q$ called the \emph{normal metric}. If the diffeomorphism group is equipped with a right-invariant metric, then the normal metric associated with the group action is a natural choice for the metric on shape space. For discussions of normal metrics induced by actions of the diffeomorphism group on smooth manifolds in the context of computational anatomy, see e.g. \cite{Younes2009S40, You2010}. 

The objective in computational anatomy is to quantify the distance between two given shapes by computing the length of the geodesic path between them, with respect to the normal metric on shape space. This is equivalent to computing a (horizontal) geodesic path on the group of diffeomorphisms that carries one shape into the other. This \emph{horizontal lifting property of geodesics} has been crucial in the understanding and the numerics of LDM, and it has clarified the close connection between ideal fluid mechanics and image registration. Namely, geodesic flows on diffeomorphism groups are described by Euler--Poincar\'e equations, called \emph{EPDiff} equations in the diffeomorphism context \cite{HoMaRa1998,MaRa1994,Younes2009S40}. Geodesic flows on the subgroup of volume-preserving diffeomorphisms were famously  identified with Euler's equations for ideal incompressible fluid flow in \cite{Ar1966}. Paper \cite{MiTrYo2006} studies momentum conservation properties of geodesic EPDiff flows in the context of image matching. In particular, these geodesic flows are encoded by their initial momenta, and horizontality means that only momenta of a specific form are permitted. For example, landmark-based geodesic image matching naturally summons the singular momenta that were introduced as solitons for shallow water waves on the real line in \cite{CaHo1993} and then characterized as singular momentum maps in any number of  dimensions in \cite{HoMa2004}. We refer to \cite{HoRaTrYo2004, MiTrYo2006, CoHo2010, Younes2009S40, via09, BrFGBRa2010, FGBRa2011} for further details. 

In addition to their importance in computational anatomy, normal metrics and their geodesics also appear in a range of problems in mechanics and control theory, see for example \cite{BlCr1996, HoMa2004, NiAk2005}.

\paragraph{Trajectory planning and longitudinal studies in computational anatomy.} Longitudinal studies in computational anatomy seek to determine a path that interpolates optimally through a \emph{time-ordered series} of images, or shapes. This sort of task is  also familiar in data assimilation. Depending on the specific application, the interpolant will be required to have a certain degree of spatiotemporal smoothness. For example, the pairwise geodesic matching procedure can be extended to piecewise-geodesic interpolation through several shapes, as in \cite{BegKhanISBI08, DurrlemanEtAl2009}. If a higher degree of smoothness is required, then one must investigate higher-order interpolation methods.  

Higher-order interpolation methods on finite-dimensional spaces have been studied extensively in the context of trajectory planning in aeronautics, robotics, computer graphics, biomechanics and air traffic control. In particular, the study of Riemannian cubics in manifolds with curvature and their higher-order generalizations originated in \cite{GaKa1985}, \cite{NoHePa1989} and \cite{CrSL1995}. 
Riemannian cubics are solutions of Euler-Lagrange equations for a certain second-order variational problem in a finite-dimensional connected Riemannian manifold, to find a curve that interpolates between two points with given initial and final velocities, subject to  minimal mean-square covariant acceleration. 
The mathematical theory of Riemannian cubics was subsequently developed in a series of papers including \cite{CaSLCr1995, CrCaSLHamiltonian1998, CaSLCr2001, No2003, No2004, No2006, splinesanalyse, Krakowski2005} and Part I of the current study, \cite{Gay-BalmazEtAl2010}. The last reference treats Riemannian cubics on Lie groups by Euler--Poincar\'e reduction. Engineering applications are discussed in \cite{PaRa1997, ZeKuCr1998, HuBl2004, HuBl2004a, Noakes_Camera2003}, amongst others. Related higher-order interpolation methods have been studied, for example, in \cite{BeKu2000, HoPo2004, MaSL2004, NoPo2005,Noakes_SphericalSplines2006, MaSLeKr2010}. We refer to \cite{Popiel2007, MaSLeKr2010} and \cite{Noakes_SphericalSplines2006} for extensive references and historical discussions concerning Riemannian cubics, their higher-order generalizations, and related higher-order interpolation methods.

\subsection{Motivation} \label{Sec-Motivation}
In this paper, we consider the problem of Riemannian cubics for normal metrics, focusing on their lifting and projection properties. Recall that in the context of normal metrics one considers a manifold of objects, or shapes, that are acted upon by a Lie group of transformations. Two distinct interpolation strategies offer themselves. First, one may choose to define a variational principle on the Lie group, or indeed its Lie algebra, and find an optimal path $g(t)$ that transforms the initial shape $q$ as $q(t) = g(t) q$, such that $q(t)$ passes through the prescribed configurations. This type of higher-order interpolation was proposed in regard to applications in Computational Anatomy, in \cite{Gay-BalmazEtAl2010} and studied there in detail in the finite-dimensional setting.  Alternatively, one may define a variational principle on shape space itself and find an optimal curve that interpolates the given shapes. This was the approach of \cite{TrVi2010}, where interpolation by Riemannian cubics on shape space was proposed, and existence results in the case of landmarks were given. The particular cost functionals that interest us here are, on the group,
\begin{equation}
  S_G[g] = \int_0^1 \left\|\frac{D}{Dt}\dot{g}\right\|_g^2\, dt, \nonumber
\end{equation}
and on the object manifold,
\begin{equation}
  S_Q[q] = \int_0^1 \left\|\frac{D}{Dt}\dot{q}\right\|_q^2\, dt. \nonumber
\end{equation}
Hamilton's principle, $\delta S = 0$, leads to cubics in the respective manifolds $G$ and $Q$. One is therefore naturally led to the question of how Riemannian cubics on the group of transformations are related to those on the object manifold. This is a new question in geometric mechanics and its answer is potentially important in applications of computational anatomy. 

In this paper, we begin the investigation of this question. We first analyze horizontal lifts of cubics on the object manifold to the group of transformations. In the context of symmetric spaces, we completely characterize the class of cubics on the object manifold that can be lifted horizontally to cubics on the group of transformations. For rank-one symmetric spaces this, selects geodesics composed with cubic polynomials in time. We then study non-horizontal curves in $G$. We show that certain types of non-horizontal geodesics project to cubics in $Q$. Finally, we present the theory of second-order Lagrange--Poincar\'e reduction for Riemannian cubics in the group of transformations.  The reduced form of the equations reveals the obstruction for such a cubic to project to a cubic on the object manifold.  
\subsection{Main content of the paper}
The main content of the paper may be summarized as follows:
\begin{description}
\item In Section \ref{Section-Geometric_Setting} we outline the geometric setting for the present investigation of Riemannian cubics for normal metrics and their relation to Riemannian cubics on the Lie group of transformations. In particular, we summarize the definition of higher-order tangent bundles by following \cite{CeMaRa2001}. Then we define normal metrics and recall that the projection $G \rightarrow Q$ which maps an element of the group to the transformed image of a reference shape is a Riemannian submersion.
\item In Section \ref{Section-Covariant_derivatives} we provide the key expressions for covariant derivatives of curves, and vector fields along curves, both in Lie groups and in object manifolds with normal metric. The horizontal generator of a curve in the object manifold is introduced and expressed in terms of the momentum map of the cotangent lifted action.
\item In Section \ref{Section-Cubic_splines_for_normal_metrics} we derive the  equations of Riemannian cubics for normal metrics. For ease of exposition we first consider a more general context and then specialize to the case of Riemannian cubics. Here the horizontal generator plays a crucial role. Invariant metrics on Lie groups are a simple example of normal metrics, as are the metrics on symmetric spaces. These examples are worked out in detail. We also recall from \cite{Gay-BalmazEtAl2010} how Riemannian cubics on Lie groups can be treated equivalently by Euler--Poincar\'e reduction. Our derivation of the Euler-Lagrange equations bypasses any mention of curvature. Therefore, these equations can also be used to \emph{compute} curvatures by means of the general equation for Riemannian cubics derived in \cite{NoHePa1989, CrSL1995}. This is demonstrated by two simple examples.
\item  In Section \ref{Sec-Horizontal_lifts} we study horizontal lifting properties of Riemannian cubics. Our form of the Euler-Lagrange equations is particularly well suited for this task, due to the appearance of the horizontal generator of curves. We characterize the cubics in symmetric spaces that can be lifted horizontally to cubics in the group of isometries. We then proceed to the more general situation of a Riemannian submersion and state necessary and sufficient conditions under which a cubic on the object manifold lifts horizontally to a cubic on the Lie group of transformations.

\item In section \ref{Section-Extended_analysis_Reduction_by_isotropy_subgroup} we extend the previous considerations to include non-horizontal curves on the Lie group. We show that certain non-horizontal geodesics on the group of transformations project to cubics on the object manifold. We then reduce the Riemannian cubic variational problem on the group by the isotropy subgroup of a reference object. To achieve this, we use  higher-order Lagrange--Poincar\'e reduction \cite{CeMaRa2001, Gay-BalmazEtAl2011HOLPHP}. The reduced Lagrangian couples horizontal and vertical parts of the motion, and this explains the absence of a general horizontal lifting property for cubics. Namely, the reduced equations that describe Riemannian cubics on the Lie group contain the equation that characterizes Riemannian cubics on the object manifold, plus extra terms. These extra terms represent the obstruction  for a cubic on the Lie group to project to a cubic on the object manifold. In this sense, these reduced equations fully describe the relation between cubics on the Lie group and cubics on the object manifold. They also lend themselves to further study of the questions investigated in the present paper. 
\end{description} 
\section{Geometric setting}\label{Section-Geometric_Setting}
This section introduces the geometric ingredients used in the paper. In particular, it reviews Hamilton's principle on higher-order tangent bundles. It also introduces Riemannian cubics and some of their generalizations. It then defines normal metrics and explains how the projection $G \rightarrow Q$ mapping an element of the group to the transformed image of a reference shape $a \in Q$ is a Riemannian submersion.
\subsection{$\mathbf{k^{th}}$-order tangent bundles}
The $k^{th}$-order tangent bundle $\tau_Q^{(k)}: T^{(k)}Q \rightarrow Q$ of a smooth manifold $Q$ is defined as a set of equivalence classes of curves, as follows: Two curves $q_1(t)$, $q_2(t)$ are \emph{equivalent}, if and only if their time derivatives at $t = 0$ up to order $k$ coincide in any local chart. That is, $q_1^{(l)}(0) = q_2^{(l)}(0)$, for $l = 0, \ldots , k$. The equivalence class of a given curve $q(t)$ is denoted by $\left[q\right]^{(k)}_{q(0)}$. One then defines $T^{(k)}Q$ as the set of equivalence classes of curves, with projection
\begin{equation}
  \tau_Q^{(k)}: T^{(k)}Q \rightarrow Q\,, \quad \left[q\right]^{(k)}_{q(0)} \mapsto q(0)\,.
\end{equation}
Finally, the inverse image of $q_0 \in Q$ by $\tau^{(k)}_Q$ will be denoted by $T^{(k)}_{q_0}Q$. This is the set of equivalence classes of curves based at $q_0$. Note that $T^{(0)}Q = Q$ and $T^{(1)}Q = TQ$.

Given a curve $q(t)$ one defines the $k^{th}$-order tangent element at time $t$ to be
\begin{equation}
  \left[q\right]^{(k)}_{q(t)}:= \left[h\right]^{(k)}_{h(0)}\,, \quad \mbox{where} \quad h: \tau \mapsto q(t + \tau)\,.
\end{equation}
We will sometimes use the coordinate notation $(q(t), 
\dot{q}(t), \ldots, q^{(k)}(t))$ to denote 
$\left[q\right]^{(k)}_{q(t)}$.
For more information on higher-order tangent bundles see 
\cite{CeMaRa2001}.

A smooth map $f: M \rightarrow N$ induces a map between $k^{th}$-order tangent bundles,
\begin{equation}
T^{(k)}f : T^{(k)}M \rightarrow T^{(k)}N\,, \quad \left[q\right]_{q_0}^{(k)} \mapsto \left[f \circ q\right]_{f(q_0)}^{(k)}\,.
\end{equation}
Therefore, a group action $\Phi: G \times Q \mapsto Q$ on the base manifold lifts to a group action on the $k^{th}$-order tangent bundle,
\begin{equation}\label{Lift_of_group_action}
\Phi^{(k)}: G \times T^{(k)}Q \mapsto T^{(k)}Q \,, \quad \Phi^{(k)}_g: \left[q\right]^{(k)}_{q(0)} \mapsto \left[ \Phi_g \circ q \right]^{(k)}_{q(0)}\,.
\end{equation}
\subsection{$\mathbf{k^{th}}$-order Euler-Lagrange equations}
A $k^{th}$-order Lagrangian is a function $L: T^{(k)}Q \rightarrow \mathbb{R}$. In the higher-order generalization of Hamilton's principle, one seeks a critical point of the functional
\begin{equation}\label{Hamiltons_principle_general}
  \mathcal{J}[q] := \int_{0}^{1} L\left(q(t), \dot{q}(t), \ldots, q^{(k)}(t)\right) \, dt
\end{equation}
with respect to variations of the curve $q(t)$ satisfying fixed end point conditions $q^{(l)}(0) = q_0^{(l)}$ and $q^{(l)}(1) = q_1^{(l)}$ for $l = 0, \ldots, k-1$. A curve $q(t)$ respecting the end point conditions satisfies Hamilton's principle if and only if it satisfies the \emph{$k^{th}$-order Euler-Lagrange equations}
\begin{equation}
[L ]_q^{k}
:=
  \sum_{j=0}^k (-1)^j \frac{d^j}{dt^j} \frac{\partial L}{\partial q^{(j)}} = 0\,.
\end{equation}
We will use from now on the standard $\delta$-notation for variations. Let $\varepsilon \mapsto f_\epsilon$ be a variation of a quantity $f = f_0$. Define
\begin{equation}
  \delta f := \left.\frac{d}{d\varepsilon}\right|_{\varepsilon = 0} f_{\varepsilon}\,.
\end{equation}
Hamilton's principle, for example, then takes the simple form $\delta J = 0$.
\paragraph{Examples: Riemannian cubic polynomials and generalizations.} Riemannian cubics, as introduced in \cite{GaKa1985, NoHePa1989} and \cite{CrSL1995}, generalize cubic polynomials in Euclidean space to Riemannian manifolds. Let $(Q, \gamma)$ be a Riemannian manifold and denote by $\frac{D}{Dt}$ the covariant derivative with respect to the Levi-Civita connection $\nabla$ for the metric $\gamma$. Let $\|\cdot \|$ be the norm induced by $\gamma$. Consider Hamilton's principle \eqref{Hamiltons_principle_general} for $k=2$ with Lagrangian $L: T^{(2)}Q \rightarrow \mathbb{R}$ given by
\begin{equation}\label{Cubics_Lagrangian}
  L\left(q, \dot{q}, \ddot{q}\right) = \frac{1}{2}\left\|\frac{D}{Dt} \dot{q}\right\|_q^2 \,.
\end{equation}
This Lagrangian is indeed well-defined on the second-order tangent bundle $T^{(2)}Q$, since in coordinates
\begin{equation}
  \frac{D}{Dt} \dot{q}^k = \ddot{q}^k + \Gamma^k_{ij}(q)\dot{q}^i \dot{q}^j\,,
\end{equation}
where $\Gamma^{k}_{ij}(q)$ are the Christoffel symbols at the point $q$. Denoting by $R$ the curvature tensor defined by 
$R(X, Y) Z := \nabla_X \nabla_Y Z - \nabla_Y \nabla_X Z - \nabla_{[X, Y]}Z$ for any vector fields $X,Y,Z \in 
\mathfrak{X}(Q)$, the Euler-Lagrange equation is
\begin{equation}\label{ELeqns-T1}
  \frac{D^3}{Dt^3}\dot q(t)+ R \left( \frac{D}{Dt}\dot q(t),\dot q(t)
\right) \dot q(t)= 0.
\end{equation}
A solution of this equation is called a \emph{Riemannian cubic}, or \emph{cubic} for short. These are the curves we shall study in this paper. 

We also mention two generalizations of Riemannian cubics. The first one consists of the class of 
\emph{geometric $k$-splines} \cite{CaSLCr1995} for $k \geq 2$ with Lagrangian $L: T^{(k)}Q \rightarrow \mathbb{R}$,
\begin{equation}
  L\left(q, \dot{q}, \ldots, q^{(k)}\right)  =  \frac{1}{2}\left\|\frac{D^{k-1}}{Dt^{k-1}} \dot{q}\right\|_q^2.
\end{equation}
Note that the case $k=2$ recovers the Riemannian cubics. The Euler-Lagrange equations are  \cite{CaSLCr1995}
\begin{equation}
\frac{D^{2k-1}}{Dt^{2k-1}}\dot q(t) + \sum_{j=2}^k (-1)^j R
\left(\frac{D^{2k-j-1}}{Dt^{2k-j-1}}\dot q(t),
\frac{D^{j-2}}{Dt^{j-2}}\dot q(t)\right)\dot q(t) =0.
\end{equation}
The second generalization comprises the class of \emph{cubics in tension}; see, for example, \cite{HuBl2004}.
\subsection{Normal metrics}\label{Section_Normal_metrics}
\paragraph{Group actions.} Let $G$ be a Lie group with Lie algebra $\mathfrak{g}$, acting from the \textit{left} on a smooth manifold $Q$. We denote the action  by
\begin{equation}
\Phi:G \times Q \rightarrow Q\,, \quad (g, q) \mapsto gq:= \Phi_g(q).
\end{equation}
The infinitesimal generator of the action corresponding to $\xi \in \mathfrak{g}$ is the vector field on $Q$ given by
\begin{equation}
  \xi_Q(q) := \left.\frac{d}{dt}\right|_{t = 0}\operatorname{exp}( t\xi)q.
\end{equation}
 In accordance with \eqref{Lift_of_group_action}, the tangent lift of $\Phi$ is defined as the action of $G$ on $TQ$,
\begin{equation}
G \times TQ \rightarrow TQ, \quad (g, v_q) \mapsto gv_q:= T\Phi_g(v_q),
\end{equation}
with infinitesimal generator $\xi_{TQ}$ corresponding to $\xi \in \mathfrak{g}$. Note that we have the relation
\begin{equation}\label{relation_xiQ_xiTQ}
T \tau _Q ( \xi _{TQ}( v _q ))= \xi _Q (q),
\end{equation} 
where $ \tau _Q :TQ \rightarrow Q$ is the tangent bundle projection. Similarly, one defines the cotangent lifted action as
\begin{equation}
  G \times T^*Q \rightarrow T^*Q, \quad (g, \alpha_q) \mapsto g\alpha_q := (T\Phi_{g^{-1}})^*(\alpha_q).
\end{equation}
The momentum map $\mathbf{J} : T^*Q \rightarrow \mathfrak{g}^*$ associated with the cotangent lift of $\Phi$ is determined by
\begin{equation}\label{Definition_momentum_map}
\left<\mathbf{J} (\alpha_q), \xi\right>_{\mathfrak{g}^* \times \mathfrak{g}}=\left<\alpha_q, \xi_Q(q)\right>_{T^*Q \times TQ},
\end{equation}
for arbitrary $\alpha_q \in T^*Q$ and $\xi \in \mathfrak{g}$.
\paragraph{Normal metrics.} Let $G$ be a Lie group acting transitively from the left on a smooth manifold $Q$. Let $\gamma_G$ be a right-invariant Riemannian metric on $G$. We will now use the action of $G$ on $Q$ in order to induce a metric $\gamma_Q$ on $Q$. To do this, define a pointwise inner product on tangent spaces $T_qQ$ by
  \begin{equation}\label{Induced_metric}
   \gamma_Q(v_q, v_q) := \min_{\left\{\xi \in \mathfrak{g} \,\left|\,\xi_Q( q) = v_q \right.\right\}} \left\{\gamma_G(\xi, \xi)\right\}\,.
 \end{equation}

We refer to \cite{You2010} for a rigorous treatment of the infinite dimensional case of diffeomorphism groups.
 We define the \emph{vertical subspace of $\mathfrak{g}$ at $q$} as
 \begin{equation}\label{Vertical_subspace_Lie_algebra_definition}
   \mathfrak{g}^V_q = \left\{ \xi \in \mathfrak{g} \big| \xi_Q(q) = 0\right\}
 \end{equation}
and the \emph{horizontal subspace} as the orthogonal complement $\mathfrak{g}_q^H = \left(\mathfrak{g}^V_q\right)^{\perp}$. Denote the orthogonal projection onto $\mathfrak{g}^H_q$ by  $\xi \mapsto \operatorname{H}_q\left(\xi\right)$. This projection depends smoothly on $q \in Q$. The vertical projection is similarly written as $\xi \mapsto \operatorname{V}_q\left(\xi\right)$. Let $\nu_1, \ldots \nu_k$ be an orthonormal basis of $\mathfrak{g}_q^V$. For $v_q$ in $T_qQ$ and $\xi$ any generator of $v_q$, i.e. $\xi_Q(q) = v_q$, we can write
\begin{align}
  \gamma_Q(v_q, v_q) &= \min_{\lambda^i \in \mathbb{R}}\left\{\gamma_G\left(\operatorname{H}_q\left(\xi\right) + \lambda^i\nu_i, \operatorname{H}_q\left(\xi\right) + \lambda^j \nu_j\right)\right\} =  \min_{\lambda^i \in \mathbb{R}}\left\{\gamma_G\left(\operatorname{H}_q\left(\xi\right), \operatorname{H}_q\left(\xi\right)\right) + \sum^k_{i=1} (\lambda^i)^2\right\} \nonumber \\
&= \gamma_G\left(\operatorname{H}_q\left(\xi\right), \operatorname{H}_q\left(\xi\right)\right). \label{Induced_metric_expression}
\end{align}
The pair $(Q, \gamma_Q)$ with $\gamma_Q$ defined pointwise by \eqref{Induced_metric} is therefore a Riemannian manifold. The metric $\gamma_Q$ is usually called a \emph{normal metric} or a \emph{projected metric}. It coincides with the normal metric considered in \cite{FGBRa2011}. 
\paragraph{Tangent action of the vertical space.}
 From 
\eqref{relation_xiQ_xiTQ}, it follows that $T \tau _Q( \xi _{TQ}( v_q))=0_q$, for all $ \xi \in \mathfrak{g}  ^V _q $. 
Let $V_{v_q} TQ : = \ker (T_{v_q} \tau_Q)$ be the vertical
space at $v_q$. Then $\xi_{TQ}(v_q) \in V_{v_q}(TQ)$  can be identified with an element of $T_qQ$ in the standard way: 
$\tau: V_{v_q}(TQ)\ni  \xi_{TQ}( v_q) \stackrel{\sim} \mapsto w_q \in T_qQ$, 
where $w_q $ is the unique vector satisfying
$\xi_{TQ}(v_q) = \left.\frac{d}{d\varepsilon}\right|_{\varepsilon = 0} (v_q + \varepsilon w_q)$. Let $\gamma_Q$ be a Riemannian metric on $Q$ and define the \emph{Connector} of $\gamma_Q$ to be the intrinsic map $K: T(TQ) \rightarrow TQ$ given in coordinates as
\begin{equation}
  K_{\operatorname{loc}}(x, w, u, v):=(x, v + \Gamma(x)(w,u)),
\end{equation}
where $\Gamma$ are the Christoffel symbols of the metric. More information on the connector can be found, for example, in \cite[\S13.8]{Michor2008}. Its key property is that for any two vector fields $X, Y \in \mathfrak{X}(Q)$,
\[
  \nabla_YX = K \circ TX \circ Y.
\]
Moreover, its restriction to vertical spaces $V_{v_q} \subset T(TQ)$ corresponds to $\tau$, as $K_{\operatorname{loc}}(x, e, 0, v) = (x, v) = \tau_{\operatorname{loc}}(x, e, 0, v)$. For $\xi \in \mathfrak{g}_q^V$ one therefore obtains
\begin{equation}\label{Konnector_formula_vertical_action}
  \tau\left(\xi_{TQ}(v_q)\right) =  K\left(\xi_{TQ}(v_q)\right) =  K\left(T\xi_Q(q)(v_q)\right) = \nabla_{v_q} \xi_Q.
\end{equation}
\paragraph{Riemannian submersion.} For $(G, \gamma_G)$ and $(Q, \gamma_Q)$ as above, fix $a \in Q$, and consider the principal bundle projection
\begin{equation}
\Pi : G \rightarrow Q\,, \quad g \mapsto ga\,.
\end{equation}
For all $g\in G$, we decompose the tangent space $T_{g}G$ into the vertical space $T^V_{g}G := \ker T_{g}\Pi$ and the horizontal space $T^H_{g}G := \left(T^V_{g}G\right)^\perp$ determined by the Riemannian metric $ \gamma _G $. These spaces are translations of appropriate subspaces of $\mathfrak{g}$,
\begin{align}
  T^V_{g}G = TL_g(\mathfrak{g}^V_a) =TR_g( \mathfrak{g}^V_q), \quad \mbox{and} \quad   T^H_{g}G = TR_g (\mathfrak{g}^H_q) \,,\label{Relation_horizontal_spaces}
\end{align}
where $q=\Pi(g)$.
The second equality in \eqref{Relation_horizontal_spaces}  is due to the right-invariance of $\gamma_G$. This justifies the terminology vertical and horizontal subspaces for the isotropy subalgebra $ \mathfrak{g}  _q ^V$ at $q$ and its orthogonal complement $ \mathfrak{g}  ^H _q $. We also record that
\begin{equation}\label{Commutativity}
T_g \Pi(\xi_G(g)) = \xi_Q (\Pi(g))\,, \quad \mbox{for all } g \in G, \xi \in \mathfrak{g}\,.
\end{equation}
It is well known that $\Pi:G \rightarrow Q$ is a Riemannian submersion, i.e., $\Pi$ is a surjective submersion and for any $v_g, w_g \in T_g^HG$, we have
\begin{equation}
  \gamma_G(v_g, w_g) = \gamma_Q\left( T_g \Pi(v_g), T_g \Pi(w_g)\right) \,.
\end{equation}

This property will be useful when we compute covariant derivatives for normal metrics in the next section.
\section{Covariant derivatives}\label{Section-Covariant_derivatives}
The main goal of this section is to obtain expressions for the covariant derivative of a curve $q(t) \in Q$, where $Q$ is equipped with a normal metric. The strategy is the following. First we compute covariant derivatives of curves in Lie groups with right-invariant metrics. Then we exploit the fact that the projection mapping $\Pi$ introduced in Section \ref{Section_Normal_metrics} is a Riemannian submersion.
\subsection{Preliminary remarks}
Let $(Q, \gamma_Q)$ be a Riemannian manifold, and let $q(t)$ be a curve in $Q$. On an open set $U \subset Q$ introduce a coordinate map $\phi: U \ni q\mapsto x \in  \mathbb{R}^n$ and let $V(t) = v^i(t) \left.\frac{\partial}{\partial x^i}\right|_{q(t)} \in T_{q(t)}Q$ be a vector field along $q(t)$. The covariant derivative, with respect to the Levi-Civita connection, of $V$ along $q(t)$ is
\begin{equation}\label{Covariant_derivative_vf_along_curve}
  \frac{D}{Dt}V = (\dot{v}^k + \Gamma^k_{ij}\dot{x}^i v^j) \left.\frac{\partial}{\partial x^k}\right|_q,
\end{equation}
where the $\Gamma^k_{ij}$ are the Christoffel symbols of $\gamma_Q$. The geodesic equation $\frac{D}{Dt}\dot{q} = 0$ becomes, in coordinates,
\begin{equation}\label{Geodesic_equation_coordinates}
   \ddot{x}^k + \Gamma^k_{ij}\dot{x}^i \dot{x}^j = 0.
\end{equation}
Let $g(t) \in G$ be a curve in a Lie group $(G, \gamma_G)$ with right-invariant Riemannian metric $\gamma_G$. The symmetry reduced form of the geodesic equation is the Euler--Poincar\'{e} equation \cite{MaRa1994}
\begin{equation}\label{Euler_Poincare_equation}
  \dot{\xi} + \operatorname{ad}^\dagger_\xi \xi = 0, \quad \mbox{with} \quad \dot{g} = \xi_G(g),
\end{equation}
 where $\operatorname{ad}^\dagger$ is the metric adjoint, i.e.,
 it has the expression 
 \begin{equation}\label{ad-dagger-def}
 \operatorname{ad}^\dagger_{\nu}{\kappa} := 
 (\operatorname{ad}^*_\nu (\kappa^\flat))^\sharp
 \end{equation}
 for any $\nu, \kappa\in \mathfrak{g}$.

We also recall a formula for the covariant derivative of horizontal vector fields for Riemannian submersions, which we will use subsequently. Let $\Pi: (\tilde{Q}, \gamma_{\tilde{Q}}) \rightarrow (Q, \gamma_Q)$ be a Riemannian submersion and denote the covariant derivatives with respect to the Levi-Civita connections on $\tilde{Q}$ and $Q$ by $\tilde{\nabla}$ and 
$\nabla$, respectively. Let $\tilde{X}, \tilde{Y} \in 
\mathfrak{X}(\tilde{Q})$ be the horizontal lifts of 
$X, Y \in \mathfrak{X}(Q)$, respectively.  Then (see e.g., 
\cite{Lee1997}),
 \begin{equation}\label{Covariant_derivative_on_RS}
   \tilde{\nabla}_{\tilde{X}} \tilde{Y} = \widetilde{\nabla_{X}Y} + \frac{1}{2} [\tilde{X}, \tilde{Y}]^V,
 \end{equation}
where the superscript $V$ denotes the vertical part. The \emph{horizontal lifting property} of geodesics follows. 
Namely, if $\tilde{q}(t) \in \tilde{Q}$ is the horizontal lift of a geodesic $q(t) \in Q$, that is, $\nabla_{\dot{q}} \dot{q} = 0$, then $\tilde{q}(t)$ is a geodesic, since $ \tilde{\nabla}_{\dot{\tilde{q}}}\dot{\tilde{q}} = \widetilde{\nabla_{\dot{q}}\dot{q}} = 0$. Note that applying $T\Pi$ to both sides of \eqref{Covariant_derivative_on_RS} gives
\begin{equation}\label{Projected_covariant_derivative_on_RS}
  T\Pi \left(\tilde{\nabla}_{\tilde{X}} \tilde{Y}\right) = \nabla_XY.
\end{equation}
\subsection{Covariant derivatives for normal metrics} 
The following proposition is a compilation of well-known expressions that will be used extensively in the rest of the paper. The proofs of \eqref{Covariant_derivative_group_prop} -- \eqref{Covariant_derivative_covf_prop} can be found, for example, in \cite{KrMi1997}. We note that the expression \eqref{Momentum_map_horizontal_generator_prop} below for the horizontal generator was also used in \cite{TrVi2010} and \cite{FGBRa2011}.
\begin{proposition}\label{Covariant_derivative_main_proposition}
  Let $(G, \gamma_G)$ be a Lie group with right-invariant metric, acting transitively from the left on a manifold $(Q, \gamma_Q)$ with normal metric $\gamma_Q$.
  \begin{enumerate}[{\bf (i)}]
  \item Let $g(t)$ be a curve in $G$, and define $\xi(t) \in \mathfrak{g}$ by $\dot{g} = \xi_G(g)$. Then,
    \begin{equation}\label{Covariant_derivative_group_prop}
      \frac{D}{Dt} \dot{g} = \left( \dot{\xi} + \operatorname{ad}^\dagger_\xi \xi\right) _G(g).
    \end{equation}
  \item More generally, let $V(t) \in T_{g(t)}G$ be a vector field along a curve $g(t) \in G$. Define curves $\xi(t), \, \nu(t) \in \mathfrak{g}$ by $\dot{g} = \xi_G(g)$ and $V = \nu_G(g)$, respectively. Then,
    \begin{equation}\label{Covariant_derivative_vf_prop}
      \frac{D}{Dt} V =\Big( \dot{\nu} + \frac{1}{2} \operatorname{ad}^\dagger_\xi \nu + \frac{1}{2} \operatorname{ad}^\dagger_\nu \xi - \frac{1}{2} [\xi, \nu]\Big) _G(g).
    \end{equation}
Furthermore, let $M(t) \in T^*_{g(t)}G$ be a covector field along $g(t)$ and define $\mu(t) \in \mathfrak{g}^*$ by $\mu = (TR_g)^* M$. Then,
\begin{equation}\label{Covariant_derivative_covf_prop}
  \frac{D}{Dt}M = (TR_{g^{-1}})^* \Big(\dot{\mu} - \frac{1}{2} (\operatorname{ad}_\xi(\mu^\sharp))^\flat + \frac{1}{2} \operatorname{ad}^*_{\mu^{\sharp}}\xi^\flat + \frac{1}{2} \operatorname{ad}^*_\xi \mu\Big).
\end{equation}
\item Let $q(t)$ be a curve in $Q$ and let $\xi(t) \in \mathfrak{g}$ be a curve satisfying $\dot{q} = \xi_Q(q)$. Then
\begin{equation}\label{Covariant_derivative_induced_prop}
\frac{D}{Dt} \dot{q} = \left( \dot{\xi} + \operatorname{ad}^\dagger_{\operatorname{H}_q\left(\xi\right)} \operatorname{H}_q\left(\xi\right)\right) _Q(q) + \nabla_{\dot{q}}\big(\operatorname{V}_q\left(\xi\right)\big)_Q,
\end{equation}
In particular, if $\xi(t) \in \mathfrak{g}^H_{q(t) }$ is the unique horizontal generator of $q(t)$, then
\begin{equation}\label{Covariant_derivative_induced_horizontal_prop}
\frac{D}{Dt} \dot{q} = \left( \dot{\xi} + \operatorname{ad}^\dagger_\xi \xi\right) _Q(q).
\end{equation}
\item Let $q(t)$ be a curve in $Q$. The unique horizontal generator of $q(t)$ is given by the Lie algebra element $\bar{\mathbf{J}}(\dot{q})$ defined by
\begin{equation}\label{Momentum_map_horizontal_generator_prop}
\bar{\mathbf{J}}(\dot{q}) := \left( \mathbf{J} (\dot{q}^\flat)\right) ^\sharp\in \mathfrak{g}  ^H _q ,
\end{equation}
where $\mathbf{J} $ is the cotangent lift momentum map defined in Section \ref{Section_Normal_metrics}. In particular,
\begin{equation}\label{Covariant_derivative_momentum_map_prop}
\frac{D}{Dt} \dot{q} =  \left(\partial _t \bar{\mathbf{J}}( \dot q) + \operatorname{ad}^\dagger_{\bar{\mathbf{J}}( \dot q)} \bar{\mathbf{J}}( \dot q) \right)_Q(q).
\end{equation}
\end{enumerate}
\end{proposition}
\begin{proof}
We refer to \cite{KrMi1997} for the proof of {\bf{(i)}} and {\bf{(ii)}}. In order to show {\bf{(iii)}} recall the projection mapping $\Pi: G \mapsto Q$, $g \mapsto ga$, for a fixed $a \in Q$. Let $q(t)$ be a curve in $Q$ and define the curve $\xi(t) \in \mathfrak{g}$ to be its horizontal generator, that is, $\dot{q} = \xi_Q(q)$ and $\xi \in \mathfrak{g}^H_q$. Choose $g_0 \in \Pi^{-1}(q(0))$ and define $g(t) \in G$ by $g(0) = g_0$ and $\dot{g} = \xi_G(g)$. Then $g(t)$ is the horizontal lift of $q(t)$ through $g _0 $. We apply $T\Pi$ to \eqref{Covariant_derivative_on_RS} and use \eqref{Covariant_derivative_group_prop} and \eqref{Commutativity} to find
\begin{equation}
\frac{D}{Dt} \dot{q} = \nabla_{\dot{q}}\dot{q} = T_g \Pi\left( \tilde{\nabla}_{\dot{g}}\dot{g}\right)  = T_g \Pi\left( (\dot{\xi} + \operatorname{ad}^\dagger_\xi \xi)_G(g)\right)  = (\dot{\xi} + \operatorname{ad}^\dagger_\xi \xi)_Q (q).
\end{equation}
This shows \eqref{Covariant_derivative_induced_horizontal_prop}.

Consequently, we have
\[
\frac{D}{Dt} \dot{q} = \left(\partial _t H_q (\xi) + \operatorname{ad}^\dagger_{\operatorname{H}_q\left(\xi\right)} \operatorname{H}_q\left(\xi\right)\right) _Q(q)= \left(\partial _t \xi + \operatorname{ad}^\dagger_{\operatorname{H}_q\left(\xi\right)} \operatorname{H}_q\left(\xi\right)\right) _Q(q) - \left( \partial _t \operatorname{V}_q(\xi) \right) _Q (q),
\]
so it remains to show that $\left( \partial _t V_q(\xi) \right) _Q (q)=-\nabla_{\dot{q}}\big(\operatorname{V}_q\left(\xi\right)\big)_Q$.

We first show that for $\eta (t) \in \mathfrak{g}  ^H _{q(t)}$ and with $q=q(0)$, we have
\begin{equation} 
\left( \left.\frac{d}{dt}\right|_{t=0} \eta (t) \right) _Q (q)= \left.\frac{D}{Dt}\right|_{t=0} \left( \eta (t) \right) _Q (q) .
\label{desiredformula1}
\end{equation} 
Writing $ \eta (t) = \eta (0)+ \nu (t)$, we have $ \eta (0)_Q(q)=0$, $ \nu (0) =0$ and we compute
\begin{align*} 
\left.\frac{D}{Dt}\right|_{t=0} \left( \eta (t) \right) _Q(q)&= \left.\frac{D}{Dt}\right|_{t=0} \left( \eta (0)+ \nu (t)  \right) _Q(q)= \left.\frac{D}{Dt}\right|_{t=0} \left( \eta (0) \right) _Q(q) + \left( \nu (t)  \right) _Q(q)\\
&= \left.\frac{D}{Dt}\right|_{t=0} \left(\nu  (t) \right) _Q(q)= K \left( \left.\frac{d}{dt}\right|_{t=0}  \left(\nu  (t) \right) _Q(q)\right)\\
& = K \left(\left.\frac{d}{dt}\right|_{t=0} \left.\frac{d}{d\varepsilon}\right|_{\varepsilon=0} \Phi _{ \operatorname{exp}( \varepsilon \nu (t) )}(q)  \right)=K \left( \left.\frac{d}{d\varepsilon}\right|_{\varepsilon=0}\left.\frac{d}{dt}\right|_{t=0} \Phi _{ \operatorname{exp}( \varepsilon \nu (t) )}(q)  \right) \\
&= K \left( \left.\frac{d}{d\varepsilon}\right|_{\varepsilon=0}\left(  \varepsilon \dot{\nu} (0) \right) _Q (q)  \right)= \left( \dot{\nu} (0) \right) _Q (q)= \left( \left.\frac{d}{dt}\right|_{t=0} \eta  (t) \right) _Q (q) ,
\end{align*} 
where in the seventh equality we used that $ \nu (0) =0$. This demonstrates equation (\ref{desiredformula1}).

We now show that 
\begin{equation} 
\left.\frac{D}{Dt}\right|_{t=0} \left( \eta (t) \right) _Q(q)= - \left.\frac{D}{D\varepsilon}\right|_{\varepsilon=0} \left( \eta (t) \right) _Q (q(t+ \varepsilon ))= -\nabla _{\dot q} \eta _Q
\,,\quad\hbox{for}\quad
\eta (t) \in \mathfrak{g}  ^H _{q(t)}
\label{desiredformula2}
\end{equation} 
We have
\begin{align*} 
\eta (t)_{TQ}(\dot q(t))&= \left.\frac{d}{d\varepsilon}\right|_{\varepsilon=0} T \Phi _{\operatorname{exp}( \varepsilon \eta (t))} \left( \frac{d}{dt} q(t) \right)  \\
&= \left.\frac{d}{d\varepsilon}\right|_{\varepsilon=0} \left[ \frac{d}{dt}\Phi _{\operatorname{exp}( \varepsilon \eta  (t))}(q(t))- \left.\frac{d}{ds}\right|_{s=0}  \Phi _{\operatorname{exp}( \varepsilon \eta  (t+s))}(q(t))\right] \\
&= \kappa _{TTQ} \left( \frac{d}{dt}\left.\frac{d}{d\varepsilon}\right|_{\varepsilon=0} \Phi _{\operatorname{exp}( \varepsilon \eta  (t))}(q(t))-   \left.\frac{d}{ds}\right|_{s=0} \left.\frac{d}{d\varepsilon}\right|_{\varepsilon=0}\Phi _{\operatorname{exp}( \varepsilon \eta  (t+s))}(q(t))\right)\\
&=\kappa _{TTQ} \left( \frac{d}{dt}(\eta  (t))_Q(q(t))-\left.\frac{d}{ds}\right|_{s=0} \eta  (s+t)_Q(q(t)) \right)\\
& =-\kappa _{TTQ} \left( \operatorname{Ver}_{0_{q(t)}} \left.\frac{D}{Ds}\right|_{s=0}  \eta  (s+t)_Q(q(t)) \right) =-\operatorname{Ver}_{\dot q(t)} \left.\frac{D}{Ds}\right|_{s=0}  \eta  (s+t)_Q(q(t)),
\end{align*} 
where $ \kappa _{TTQ}: TTQ \rightarrow TTQ $ denotes the canonical involution, which in local coordinates is expressed as
\[
\kappa _{TTQ}(q,\dot{q},\delta{q},\delta\dot{q}) = (q,\delta{q},\dot{q},\delta\dot{q})
\,.\]
Since $K \left( \eta (t)_{TQ}(\dot q(t)) \right) = K \left( T\eta (t)_{Q}(\dot q(t)) \right) =\left.\frac{D}{D\varepsilon}\right|_{\varepsilon=0} \left( \eta (t) \right) _Q (q(t+ \varepsilon ))$ and $K \circ \operatorname{Ver}_{v _q }= id_{TQ}$, we obtain the desired formula in (\ref{desiredformula2}) above.

Combining the results in (\ref{desiredformula1}) and (\ref{desiredformula2}), we get $\left( \partial _t V_q(\xi) \right) _Q (q)=-\nabla_{\dot{q}}\big(\operatorname{V}_q\left(\xi\right)\big)_Q$ which proves claim {\bf {(iii)}} of the proposition.

The final step is to prove {\bf{(iv)}}. For a fixed $q \in Q$, arbitrary $w_q, \, v_q \in T_qQ$ and $\xi \in \mathfrak{g}$ with $\xi_Q (q) = v_q$ we obtain the following chain of equalities; note that $(\mathbf{J} (w_q^\flat))^\sharp \in \mathfrak{g}_q^H$ is horizontal because of \eqref{Definition_momentum_map}.
\begin{align}
  \gamma_Q\big(((\mathbf{J} (w_q^\flat))^\sharp)_Q(q), v_q\big) &= \gamma_G((\mathbf{J} (w_q^\flat))^\sharp, \xi) = \left<\mathbf{J} (w_q^\flat), \xi\right>_{\mathfrak{g}^*\times \mathfrak{g}} \\ &= \left<w_q^\flat, \xi_Q(q)\right>_{T^*Q \times TQ} = \gamma_Q(w_q, v_q).
\end{align}
Since $v_q$ was arbitrary we conclude $w_q = ((\mathbf{J} (w_q^\flat))^\sharp)_Q(q)$. This, together with \eqref{Covariant_derivative_induced_horizontal_prop}, shows {\bf{(iv)}}.
\end{proof}
\section{Cubics for normal metrics }\label{Section-Cubic_splines_for_normal_metrics}
 In this section we derive the  equations of Riemannian cubics for normal metrics. For ease of exposition we first consider a more general context and then particularize to the case of Riemannian cubics. The examples of Lie groups and of symmetric spaces are worked out in detail. 
\subsection{Preparations}
Consider a manifold $Q$ with a linear connection on its tangent bundle $TQ$. Denote the covariant derivative with respect to this linear connection by $\frac{D}{Dt}$. For $v_q, \, w_q \in T_qQ$ write $(v_q)^H_{w_q} \in T_{w_q}TQ$ for the horizontal lift of $v_q$ to $w_q$, i.e., in local coordinates,
\[
(q,v)^H_{(q,w)} = (q,w,v, -\Gamma(q)(w,v)),
\]
where $\Gamma$
is the Christoffel map of the linear connection. The vertical lift of $v_q$ to $w_q$ is written $(v_q)^V_{w_q} := \left.\frac{d}{d\varepsilon}\right|_{\varepsilon = 0} (w_q + \varepsilon v_q)$. For a variation $(t, s) \mapsto q(t, s)$ of a curve $q(t) = q(t, 0)$ the curve $\delta \dot{q}(t) := \left.\frac{d}{ds}\right|_{s = 0} \dot{q}(t, s) \in T_{\dot q(t) }(TQ)$ splits into horizontal and vertical parts
\begin{equation}\label{Splitting_delta_q_dot}
  \delta \dot{q} = (\delta q)^H_{(q, \dot{q})} + \Big(\left.\frac{D}{Ds}\right|_{s=0} \dot{q}\Big)^V_{(q, \dot{q})}.
\end{equation}
For a function $\xi: TQ \rightarrow \mathfrak{g}$ and an arbitrary $v_q \in TQ$ define the $\mathfrak{g}$-valued linear form $\frac{\delta \xi}{\delta q}\big|_{v_q}$ by
\begin{equation}
  \Big<\left.\frac{\delta \xi}{\delta q}\right|_{v_q}, w_q\Big> := \left.\frac{d}{d\varepsilon}\right|_{\varepsilon = 0} \xi(v(\varepsilon)), \quad \mbox{for any} \quad w_q \in T_qQ,
\end{equation}
where $v(\varepsilon)$ is any curve in $TQ$ with $\left.\frac{d}{d\varepsilon}\right|_{\varepsilon = 0} v(\varepsilon) = (w_q)^H_{v_q}$. On the other hand we write $\frac{\delta \xi}{\delta \dot{q}}\big|_{v_q}$ for the fiber derivative of $\xi$ at $v_q$. Note that if $q(t, s)$ is a variation of a curve $q(t) = q(t, 0)$, then using the splitting \eqref{Splitting_delta_q_dot} we get
\begin{equation}\label{delta_xi_geometric}
  \delta \xi = \Big<\frac{\delta \xi}{\delta q}, \delta q\Big> + \Big<\frac{\delta \xi}{\delta \dot{q}}, \left.\frac{D}{Ds}\right|_{s=0} \dot{q}\Big>.
\end{equation}
We define the operator $\big(\frac{\delta \xi}{\delta q}\big|_{v_q}\big)^*: \mathfrak{g}^* \rightarrow T_q^*Q$ by
\begin{equation}
  \left<\left(\left.\frac{\delta \xi}{\delta q}\right|_{v_q}\right)^* \mu, w_q\right> =  \left<\mu, \Big<\left.\frac{\delta \xi}{\delta q}\right|_{v_q}, w_q\Big>\right>\,, \quad \mbox{for any} \quad \mu \in \mathfrak{g}^*, \, w_q \in T_qQ,
\end{equation}
and similarly the operator $\big(\frac{\delta \xi}{\delta \dot{q}}\big|_{v_q}\big)^*: \mathfrak{g}^* \rightarrow T_q^*Q$ by
(keeping track of $q$ and $\dot{q}$)
\begin{equation}
  \left<\left(\left.\frac{\delta \xi}{\delta \dot{q}}\right|_{v_q}\right)^* \mu, w_q\right> = \left<\mu, \Big<\left.\frac{\delta \xi}{\delta \dot{q}}\right|_{v_q}, w_q\Big>\right> , \quad \mbox{for any} \quad \mu \in \mathfrak{g}^*, \, w_q \in T_qQ.
\end{equation}
\subsection{A generalized variational problem}
Given a Lagrangian $\ell: 2\mathfrak{g} \rightarrow \mathbb{R}$ and a smooth map $\xi : TQ \rightarrow \mathfrak{g}  $, consider the action functional on the space of curves $q(t): [0, 1] \rightarrow Q$ given by
\begin{equation}\label{Action_functional}
  \mathcal{J}[q] = \int_0^1 \ell\left( \xi(q, \dot{q}), \partial_t \xi(q, \dot{q})\right) \, dt,
\end{equation}
and Hamilton's principle $\delta \mathcal{J} = 0$ with respect to variations satisfying $\delta q(0) = \delta q(1) =0$ and $\delta \dot{q}(0) = \delta \dot{q}(1) = 0$. As we shall see below, the cubic spline Lagrangian for normal metrics fits in this framework.

Taking variations of $\mathcal{J}$ we obtain, using \eqref{delta_xi_geometric},
\begin{align*}
  \delta \mathcal{J} &= \int_0^1 \left<\frac{\delta \ell}{\delta \xi}, \delta \xi\right> + \left<\frac{\delta \ell}{\delta \dot{\xi}}, \delta \dot{\xi}\right>\, dt = \int_0^1 \left<\frac{\delta \ell}{\delta \xi} - \frac{d}{dt} \frac{\delta \ell}{\delta \dot{\xi}}, \delta \xi\right>\, dt \\
&= \int_0^1  \left<\frac{\delta \ell}{\delta \xi} - \frac{d}{dt} \frac{\delta \ell}{\delta \dot{\xi}}, \Big<\frac{\delta \xi}{\delta q}, \delta q\Big> + \Big<\frac{\delta \xi}{\delta \dot{q}}, \left.\frac{D}{Ds}\right|_{s=0} \dot{q}\Big> \right>\, dt\\
&= \int_0^1  \left<\left(\frac{\delta \xi}{\delta q}\right)^*\left(\frac{\delta \ell}{\delta \xi} - \frac{d}{dt} \frac{\delta \ell}{\delta \dot{\xi}}\right), \delta q\right> +  \left<\left(\frac{\delta \xi}{\delta \dot{q}}\right)^*\left(\frac{\delta \ell}{\delta \xi} - \frac{d}{dt} \frac{\delta \ell}{\delta \dot{\xi}}\right), \frac{D}{Dt}  \delta q\right> \, dt\\
&= \int_0^1 \left<\left[\left(\frac{\delta \xi}{\delta q}\right)^* - \frac{D}{Dt} \circ \left(\frac{\delta \xi}{\delta \dot{q}}\right)^*\right] \left(\frac{\delta \ell}{\delta \xi} - \frac{d}{dt} \frac{\delta \ell}{\delta \dot{\xi}}\right), \delta q\right>\, dt.
\end{align*}
The Euler-Lagrange equation for Hamilton's principle $\delta \mathcal{J} = 0$ is therefore 
\begin{equation}\label{Euler_Lagrange_general}
  \left[\left(\frac{\delta \xi}{\delta q}\right)^* - \frac{D}{Dt} \circ \left(\frac{\delta \xi}{\delta \dot{q}}\right)^*\right] \left(\frac{\delta \ell}{\delta \xi} - \frac{d}{dt} \frac{\delta \ell}{\delta \dot{\xi}}\right) = 0.
\end{equation}
\subsection{Cubics for normal metrics: Euler-Lagrange equations}
Let $(G, \gamma_G)$ be a Lie group with right-invariant metric $\gamma_G$, acting transitively from the left on a manifold $(Q, \gamma_Q)$ with normal metric $\gamma_Q$. Let $q(t)$ be a curve in $Q$, originating at $q_0 = q(0)$. Recall from \eqref{Momentum_map_horizontal_generator_prop} that its horizontal generator curve is given by $\bar{\mathbf{J}}(\dot{q})= (\mathbf{J}(\dot{q}^\flat))^\sharp \in \mathfrak{g}_q^H$.
\begin{lemma} For any curve $q(t)\in Q$, the curve $\partial _t (\bar{\mathbf{J}} (\dot q)) + \operatorname{ad}^\dagger_{\bar{\mathbf{J}}(\dot q)}\bar{\mathbf{J}}(\dot q)$ is horizontal, that is,  in $\mathfrak{g}^H_{q(t) }$.
\end{lemma}
\begin{proof} For any $g_0 \in \Pi^{-1}(q_0)$ define the horizontal curve $g(t) \in G$ by $g(0) = g_0$ and $\dot{g} = \left( \bar{\mathbf{J}}(\dot q)\right) _G(g)$. Then, by formula \eqref{Covariant_derivative_on_RS}, $\frac{D}{Dt} \dot{g}$ is horizontal. Moreover, by Proposition  \ref{Covariant_derivative_main_proposition},
\[
\frac{D}{Dt} \dot{g} = \left(\partial _t \bar{\mathbf{J}}(\dot q) + \operatorname{ad}^\dagger_{\bar{\mathbf{J}}(\dot q)}\bar{\mathbf{J}}(\dot q)\right) _G(g) = TR_g\left(\partial _t \bar{\mathbf{J}}(\dot q) + \operatorname{ad}^\dagger_{\bar{\mathbf{J}}(\dot q)}\bar{\mathbf{J}}(\dot q)\right).
\]
The statement of the lemma now follows from \eqref{Relation_horizontal_spaces}.
\end{proof}
This lemma enables us to rewrite the Lagrangian \eqref{Cubics_Lagrangian} of Riemannian cubics, evaluated along the curve $q(t)$, as follows,
\begin{align*}
  L(q, \dot{q}, \ddot{q}) = \frac{1}{2} \left\|\frac{D}{Dt} \dot{q}\right\|^2_q = \frac{1}{2} \left\|\left(\partial _t \bar{\mathbf{J}}(\dot q) + \operatorname{ad}^\dagger_{\bar{\mathbf{J}}(\dot q)}\bar{\mathbf{J}}(\dot q)\right)_Q(q)\right\|^2_q = \frac{1}{2} \left\|\partial _t \bar{\mathbf{J}}(\dot q) + \operatorname{ad}^\dagger_{\bar{\mathbf{J}}(\dot q)}\bar{\mathbf{J}}(\dot q)\right\|^2_{\mathfrak{g}}.
\end{align*}
Hamilton's principle \eqref{Hamiltons_principle_general} for Riemannian cubics is $\delta \mathcal{J} = 0$ with cost functional $\mathcal{J}$ of the form \eqref{Action_functional},
\begin{equation*}
  \mathcal{J}[q] = \int_0^1 \frac{1}{2} \left\|\frac{D}{Dt} \dot{q} \right\|^2_q \, dt = \int_0^1 \ell\left( \bar{\mathbf{J}}(\dot{q}), \partial_t \bar{\mathbf{J}}(\dot{q})\right) \, dt,
\end{equation*}
where
\begin{align*}
  \ell&: 2 \mathfrak{g} \rightarrow \mathbb{R}, \quad (\xi_1, \xi_2) \mapsto \frac{1}{2} \left\|\xi_2 + \operatorname{ad}^\dagger_{\xi_1} \xi_1\right\|^2_{\mathfrak{g}}.
\end{align*}
Remarkably, the function $\ell$ coincides with the reduced Lagrangian in the Euler--Poincar\'e reduction of Riemannian cubics on Lie groups, see \eqref{EP_reduced_Lagrangian} below. The variational derivatives of $\ell$ read
\begin{equation*}
  \frac{\delta \ell}{\delta \xi _1 } = (\operatorname{ad}_{\xi _1  }\eta)^\flat - \operatorname{ad}^*_\eta \xi _1 ^\flat, \quad \mbox{and} \quad \frac{\delta \ell}{\delta \xi _2 } = \eta^\flat, \qquad \text{where $\eta := \xi_2 + \operatorname{ad}^\dagger_{\xi_1} \xi_1$} 
\end{equation*}
Hence, the Euler-Lagrange equation 
\eqref{Euler_Lagrange_general} becomes
\begin{equation}\label{Spline_Euler_Lagrange_equations}
  \left[\left(\frac{\delta \bar{\mathbf{J}}}{\delta q}\right)^* - \frac{D}{Dt} \circ \left(\frac{\delta \bar{\mathbf{J}}}{\delta \dot{q}}\right)^*\right]  \left( -\partial_t \eta^\flat +  (\operatorname{ad}_{\bar{\mathbf{J}}}\eta)^\flat  - \operatorname{ad}^*_\eta \bar{\mathbf{J}}^\flat\right)  = 0, \quad \mbox{where} \quad \eta := \dot{\bar{\mathbf{J}}} + \operatorname{ad}^\dagger_{\bar{\mathbf{J}}} \bar{\mathbf{J}}.
\end{equation}
We emphasize that the covariant derivative $\frac{D}{Dt}$ is understood with respect to a chosen linear connection on the tangent bundle $TQ$. A possible choice is the Levi-Civita connection with respect to the normal metric $\gamma_Q$ on $Q$.

\subsection{Splines on Lie groups}
Riemannian cubics on Lie groups $(G, \gamma_G)$ with 
right-invariant metrics $\gamma_G$ were treated in 
\cite{Gay-BalmazEtAl2010} by second-order Euler--Poincar\'e reduction. Here we revisit the problem from the point of view 
of normal metrics.
\paragraph{Equations of motion.} We first observe that 
$\gamma_G$ is a particular case of a normal metric. Namely, let 
$(G, \gamma_G)$ act on $G$ by left multiplication. Then the normal metric induced on $G$ is again $\gamma_G$. The generator of a curve $g(t)$ is the right-invariant velocity vector, 
$\bar{\mathbf{J}}(g, \dot{g}) = TR_{g^{-1}}\dot{g}$. Choosing 
as linear connection on $TG$ the Levi-Civita connection of 
$\gamma_G$ one arrives, by Proposition 
\ref{Covariant_derivative_main_proposition} {\bf (ii)}, at
\begin{align}
&\left.\frac{\delta \bar{\mathbf{J}}}{\delta g}\right|_{(g, 
\dot{g})} : \delta g \mapsto 
-\frac{1}{2}\operatorname{ad}^\dagger_\eta \xi 
- \frac{1}{2} \operatorname{ad}^\dagger_\xi\eta 
- \frac{1}{2} [\xi, \eta], \quad \mbox{where} \quad 
\xi:= TR_{g^{-1}}\dot{g}, \, \eta := TR_{g^{-1}} \delta g,\nonumber \\
&\left(\left.\frac{\delta \bar{\mathbf{J}}}{\delta g}\right|_{(g,\dot{g})}\right)^*: \mu \mapsto (TR_{g^{-1}})^* 
\left(\frac{1}{2} \operatorname{ad}^*_{\mu^\sharp}\xi^\flat - \frac{1}{2} (\operatorname{ad}_\xi\mu^\sharp)^\flat 
- \frac{1}{2}\operatorname{ad}^*_\xi\mu\right), \quad 
\mbox{where} \quad \xi := TR_{g^{-1}}\dot{g}, \nonumber \\
&\left.\frac{\delta \bar{\mathbf{J}}}{\delta \dot{g}}\right|_{(g, \dot{g})} : v_g \mapsto TR_{g^{-1}} v_g \quad 
\mbox{for any} \quad v_g \in T_gG, \quad \mbox{and} \nonumber \\
&\left(\left.\frac{\delta \bar{\mathbf{J}}}{\delta \dot{g}}\right|_{(g, \dot{g})} \right)^*: \mu \mapsto (TR_{g^{-1}})^*\mu.\nonumber
\end{align}
Using now the expression for $\frac{D}{Dt}$ given in 
\eqref{Covariant_derivative_covf_prop} we get 
\begin{equation}
    \left[\left(\frac{\delta \bar{\mathbf{J}}}{\delta g}\right)^* - \frac{D}{Dt} \circ \left(\frac{\delta \bar{\mathbf{J}}}{\delta \dot{g}}\right)^*\right]\mu = (TR_{g^{-1}})^* \big(-\partial_t - \operatorname{ad}^*_{\bar{J}})\mu
\end{equation}
for any curves $g(t) \in G$ and $\mu(t) \in \mathfrak{g}^*$.
The Euler-Lagrange equations \eqref{Spline_Euler_Lagrange_equations} are therefore
\begin{equation}
  (TR_{g^{-1}})^* \left( \partial_t + \operatorname{ad}^*_{\bar{\mathbf{J}}}\right)  [\partial_t \eta^\flat -  (\operatorname{ad}_{\bar{\mathbf{J}}}\eta)^\flat  + \operatorname{ad}^*_\eta \bar{\mathbf{J}}^\flat] = 0, \quad \mbox{where} \quad \eta := \dot{\bar{\mathbf{J}}} + \operatorname{ad}^\dagger_{\bar{\mathbf{J}}} \bar{\mathbf{J}}.
\end{equation}
 This is equivalent to
\begin{equation}\label{Spline_equation_Lie_group_normal_metrics}
\left( \partial_t + \operatorname{ad}^*_{\bar{\mathbf{J}}}\right)  [\partial_t \eta^\flat -  (\operatorname{ad}_{\bar{\mathbf{J}}}\eta)^\flat  + \operatorname{ad}^*_\eta \bar{\mathbf{J}}^\flat] = 0, \quad \mbox{where} \quad \eta := \dot{\bar{\mathbf{J}}} + \operatorname{ad}^\dagger_{\bar{\mathbf{J}}} \bar{\mathbf{J}},
\end{equation}
and $\bar{ \mathbf{J} }:= \bar{\mathbf{J}}(g, \dot g)= \dot g g ^{-1} $.
\paragraph{Bi-invariance and the NHP equation.} If the metric $\gamma_G$ is bi-invariant, then $\operatorname{ad}^\dagger = -\operatorname{ad}$, and the above expressions simplify. Namely, $\eta = \dot{\bar{\mathbf{J}}}$ and 
\eqref{Spline_equation_Lie_group_normal_metrics} is 
equivalent to the NHP equation
\begin{equation}\label{NHP_equation}
  \dddot{\bar{\mathbf{J}}} + [\ddot{\bar{\mathbf{J}}}, \bar{\mathbf{J}}] = 0.
\end{equation}
This equation was originally derived in \cite{NoHePa1989} for the special case of $G = SO(3)$ with bi-invariant metric. It can be integrated once to yield
\begin{equation}\label{NHP_equation_integrated_once}
  \ddot{\bar{\mathbf{J}}} + [\dot{\bar{\mathbf{J}}}, \bar{\mathbf{J}}] = \nu
\end{equation}
for a constant $\nu \in \mathfrak{g}$. Solutions 
$\bar{\mathbf{J}}$ were called \emph{Lie quadratics} in 
\cite{No2003, No2004}. For $G = SO(3)$, long-term behavior and internal symmetries of Lie quadratics were studied there, both in the \emph{null} ($\nu = 0$) and the \emph{non-null} ($\nu \neq 0$) case. Generalizations to cubics in tension can be found in \cite{NoPo2005}.
\paragraph{Euler--Poincar\'{e} reduction.} We include a discussion of second-order Euler--Poincar\'e reduction and, in particular, the Riemannian cubics in this context. More details and higher-order generalizations can be found in 
\cite{Gay-BalmazEtAl2010}. The first-order case is discussed, for example, in \cite{MaRa1994}. Start with a right-invariant Lagrangian $L: T^{(2)}G \rightarrow \mathbb{R}$ with reduced Lagrangian $\ell: 2 \mathfrak{g} \rightarrow \mathbb{R}$. Consider Hamilton's principle
\begin{equation}\label{Hamiltons_Principle_on_group}
  \delta \int_0^1 L(g, \dot{g}, \ddot{g}) \, dt = 0
\end{equation}
with respect to variations of curves $g(t): [0, 1] \rightarrow G$ respecting boundary conditions $\delta g(0) = \delta g(1) = 0$ and $\delta \dot{g}(0) = \delta \dot{g} (1) = 0$. The right-invariance of $L$ leads to the equivalent reduced formulation
\begin{equation}\label{Reduced_Hamiltons_Principle_on_group}
  \delta \int_0^1 \ell(\xi, \dot{\xi}) \, dt = 0,
\end{equation}
with respect to constrained variations of curves $\xi(t): [0, 1] \rightarrow \mathfrak{g}$. Namely, one considers variations of the form $\delta \xi = \dot{\eta} - [\xi, \eta]$ with $\eta(t): [0, 1] \rightarrow \mathfrak{g}$ arbitrary up to boundary conditions $\eta(0) = \eta(1) = 0$ and $\dot{\eta} (0) = \dot{\eta}(1) = 0$. Solutions $g(t)$ of 
\eqref{Hamiltons_Principle_on_group} and solutions $\xi(t)$ of  
\eqref{Reduced_Hamiltons_Principle_on_group} are equivalent through the reconstruction relation $\xi = TR_{g^{-1}} \dot{g} 
= \bar{\mathbf{J}}(g, \dot{g})$.\newline
Taking constrained variations of \eqref{Reduced_Hamiltons_Principle_on_group} leads to the second-order Euler--Poincar\'e equation
\begin{equation}\label{Second_order_EP}
  \left(\frac{d}{dt} + \operatorname{ad}^*_\xi\right) \left(\frac{\delta \ell}{\delta \xi} - \frac{d}{dt} \frac{\delta \ell}{\delta \dot{\xi}}\right) = 0.
\end{equation}
For a curve $g(t) \in G$ with right-invariant velocity vector 
$\xi(t) = TR_{g^{-1}}\dot{g}$ we rewrite the Lagrangian of cubics,
\begin{equation*}
 L(g, \dot{g}, \ddot{g}) = \frac{1}{2} 
 \left\|\frac{D}{Dt}\dot{g}\right\|_{g}^2 
 = \frac{1}{2} \left\|(\dot{\xi}+ 
 \operatorname{ad}^\dagger_\xi\xi)_G(g)\right\|_g^2 
 = \frac{1}{2} \left\|\dot{\xi} 
 + \operatorname{ad}^\dagger_\xi\xi\right\|^2_{\mathfrak{g}},
\end{equation*}
where in the second equality we used 
\eqref{Covariant_derivative_group_prop} and the third equality follows from right-invariance of $\gamma_G$. This demonstrates that the spline Lagrangian $L$ only depends on the right-invariant velocity $\xi$ and its time-derivative $\dot{\xi}$ and is therefore right-invariant. The reduced Lagrangian can be read off as
\begin{equation}\label{EP_reduced_Lagrangian}
  \ell(\xi, \dot{\xi}) = \frac{1}{2} \left\|\dot{\xi} + \operatorname{ad}^\dagger_\xi\xi\right\|^2_{\mathfrak{g}}.
\end{equation}
The dynamics are governed by the second-order Euler--Poincar\'e equation \eqref{Second_order_EP}, which becomes, for $\ell$ as above,
\begin{equation*}
  (\partial_t + \operatorname{ad}^*_\xi) [\partial_t \eta^\flat -  (\operatorname{ad}_{\xi}\eta)^\flat  + \operatorname{ad}^*_\eta \xi^\flat] = 0, \quad \mbox{where} \quad \eta := \dot{\xi} + \operatorname{ad}^\dagger_\xi\xi.
\end{equation*}
This coincides with 
\eqref{Spline_equation_Lie_group_normal_metrics}, since 
$\xi = TR_{g^{-1}} \dot{g} = \bar{\mathbf{J}}(g, \dot{g})$.
\subsection{Splines on symmetric spaces}\label{Sec-Splines_on_symmetric_spaces}
We particularize the equation for Riemannian cubics \eqref{Spline_Euler_Lagrange_equations} to symmetric spaces. Due to the appearance of the horizontal generator, this equation lends itself to the analysis of horizontal lifting properties to be addressed in Section \ref{Sec-Horizontal_lifts}. We also comment on how it is related to the equation derived in \cite{CrSL1995}. 
\paragraph{The horizontal generator.} Recall that for any curve $g(t)$ in a Lie group $G$ with $\dot{g} = TR_g \xi$ for $\xi(t) \in \mathfrak{g}$ and any curve $\nu(t) \in \mathfrak{g}$,
\begin{equation}\label{Time_derivative_Ad}
  \frac{d}{dt} \operatorname{Ad}_{g^{-1}}\nu = \operatorname{Ad}_{g^{-1}}\left(\dot{\nu} + [\nu, \xi]\right).
\end{equation}
Let $G$ be a Lie group with a bi-invariant metric $\gamma$ that acts transitively on a manifold $Q$ equipped with the normal metric $\gamma_Q$, so that the action is by isometries. Denote by $G_a$ the isotropy subgroup of a fixed element $a \in Q$, so that $Q$ is diffeomorphic to $G/G_a$, the quotient being taken with respect to the right-action of $G_a$ on $G$.  Recall the Riemannian submersion $\Pi: G \rightarrow Q$ given by $g \mapsto ga$. Note that for any 
$g \in G$ with $ga = q$,
  \begin{align}\label{Ad_relations_projections_dash}
     \operatorname{Ad}_g \circ \operatorname{H}_a \circ \operatorname{Ad}_g^{-1} = \operatorname{H}_q \quad \mbox{and} \quad \operatorname{Ad}_g \circ \operatorname{V}_a \circ \operatorname{Ad}_g^{-1} = \operatorname{V}_q\,.
\end{align}  
\begin{lemma} Let $g(t) \in G$ be a horizontal lift of a curve $q(t) \in Q$, and let $\bar{\mathbf{J}} := \bar{\mathbf{J}}(\dot{q})$ be the horizontal generator of $q(t)$. Then,
  \begin{align}
    &\operatorname{V}_q\left(\bar{\mathbf{J}}\right) = 0, \quad \operatorname{V}_q\left(\dot{\bar{\mathbf{J}}}\right) = 0, \quad \operatorname{V}_q\left(\ddot{\bar{\mathbf{J}}} + \left[ \dot{\bar{\mathbf{J}}}, \bar{\mathbf{J}}\right]\right) = 0,\label{Horizontality_relations_1}\\
 & \operatorname{V}_q\left(\dddot{\bar{\mathbf{J}}} + 2 \left[\ddot{\bar{\mathbf{J}}}, \bar{\mathbf{J}}\right] + \left[\left[\dot{\bar{\mathbf{J}}}, \bar{\mathbf{J}}\right], \bar{\mathbf{J}}\right]\right) = 0. \label{Horizontality_relations_2}
  \end{align}
\end{lemma}
\begin{proof}
The first equation follows from horizontality of $\bar{\mathbf{J}}$. It is equivalent to $\operatorname{V}_a\left(\operatorname{Ad}_{g^{-1}}\bar{\mathbf{J}}\right) = 0$. The second equation follows from taking a time derivative of this last relation using \eqref{Time_derivative_Ad}. Note that $\dot{g} =TR_g \bar{\mathbf{J}}$. For instance, 
\[
  \partial_t \left(\operatorname{V}_a\left(\operatorname{Ad}_{g^{-1}}\bar{\mathbf{J}}\right)\right) = \operatorname{V}_a\left(\operatorname{Ad}_{g^{-1}} \dot{\bar{\mathbf{J}}}\right) = 0.
\]
Therefore,  $\operatorname{V}_q\left(\dot{\bar{\mathbf{J}}}\right) = 0$. The third and fourth equations follow from taking two more time derivatives.
\end{proof}
\paragraph{Symmetric spaces.} Assume in addition that $\sigma$ is an involutive Lie algebra automorphism of $\mathfrak{g}$ such that $\mathfrak{g}_a^V$ and $\mathfrak{g}_a^H$ are, respectively, the $+1$ and $-1$ eigenspaces. Then $(G, Q, \sigma)$ is a \emph{symmetric space structure}. The following inclusions hold for all $q \in Q$,
\begin{equation}\label{Symmetric_space_inclusions_dash}
  \left[\mathfrak{g}_q^V, \mathfrak{g}_q^V\right] \subset \mathfrak{g}_q^V, \quad  \left[\mathfrak{g}_q^H, \mathfrak{g}_q^V\right] \subset \mathfrak{g}_q^H, \quad  \left[\mathfrak{g}_q^H, \mathfrak{g}_q^H\right] \subset \mathfrak{g}_q^V.
\end{equation}
The first identity follows because $\mathfrak{g}_q^V$ is a Lie subalgebra of $\mathfrak{g}$. The second one is a consequence of the Ad-invariance of the inner product on $\mathfrak{g}$. The third one is characteristic of symmetric spaces. It follows from the eigenspace structure of $\sigma$ described above. 
 We now compute the the Euler-Lagrange equation of cubics, \eqref{Spline_Euler_Lagrange_equations}. We will find that it is equivalent to
   \begin{equation}\label{Cubic_symmetric_spaces_dash}
  \operatorname{H}_q\left(\dddot{\bar{\mathbf{J}}} + 2\left[\ddot{\bar{\mathbf{J}}}, \bar{\mathbf{J}}\right]\right) = 0.
\end{equation}
We start with a lemma.
\begin{lemma}\label{DeltaJ_deltaq}
Relative to the Levi-Civita connection on $TQ$,
  \begin{equation}
  \left<\left.\frac{\delta \bar{\mathbf{J}}}{\delta q}\right|_{(q, \dot{q})}, \delta q\right> = \left[\bar{\mathbf{J}}(\delta q), \bar{\mathbf{J}}(\dot{q})\right], \quad \mbox{and} \quad \left<\left.\frac{\delta \bar{\mathbf{J}}}{\delta \dot{q}}\right|_{(q, \dot{q})}, \delta q\right> = \bar{\mathbf{J}}(\delta q).\nonumber
  \end{equation}
\end{lemma}
\begin{proof}
The second statement is due to the linearity of $\bar{\mathbf{J}}$ on fibers of $TQ$. In order to prove the first equation, let $q(\varepsilon)$ be a curve with $q(0) = q$ and $\partial_{\varepsilon = 0} q(\varepsilon) = \delta q$, and let $g(\varepsilon)$ with $g(0) = g$ be horizontal above $q(\varepsilon)$. We construct the parallel transport of $\dot{q}$ along $q(\varepsilon)$. Define
  \begin{equation}
    \omega(\varepsilon) := \bar{\mathbf{J}}(\dot{q}) + \varepsilon \left[\bar{\mathbf{J}}(\delta q), \bar{\mathbf{J}}(\dot{q})\right], \quad X(\varepsilon) := TR_{g(\varepsilon)} \omega(\varepsilon).
  \end{equation}
To first order in $\varepsilon$,  $X(\varepsilon)$ is a horizontal vector field along $g(\varepsilon)$ and $\omega(\varepsilon)$ lies in $\mathfrak{g}_{q(\varepsilon)}^H$. Now define a vector field along $q(\varepsilon)$ by $v(\varepsilon):= T_{g(\varepsilon)}\Pi (X(\varepsilon))$. Note that to first order in $\varepsilon$, $\bar{\mathbf{J}}(v(\varepsilon)) = \omega(\varepsilon)$. Denoting by $\widetilde{\nabla}$ the covariant derivative on $G$ with respect to the Levi-Civita connection of $\gamma$, we use \eqref{Covariant_derivative_vf_prop} to get
\begin{align}
  \widetilde{\nabla}_{\delta g} X = \left(\frac{1}{2} \left[\bar{\mathbf{J}}(\delta q), \bar{\mathbf{J}}(\dot{q})\right]\right)_G(g) = \widetilde{\nabla_{ \delta q}v} + \frac{1}{2}\left[ \delta g , X\right]^V, \nonumber
\end{align}
where we used \eqref{Covariant_derivative_on_RS} in the second step. Recall that $\left[\bar{\mathbf{J}}(\delta q), \bar{\mathbf{J}}(\dot{q})\right]$ is in $\mathfrak{g}_q^V$, so that applying $T\Pi$ to the above shows $\nabla_{\delta q}v = 0$. To first order in $\varepsilon$,  $v(\varepsilon)$ is therefore the parallel transport of $\dot{q}$ along $q(\varepsilon)$, and
\begin{equation}
  \left<\left.\frac{\delta \bar{\mathbf{J}}}{\delta q}\right|_{(q, \dot{q})}, \delta q\right> = \left.\frac{d}{d\varepsilon}\right|_{\varepsilon = 0}\bar{\mathbf{J}}(v(\varepsilon))= \left.\frac{d}{d\varepsilon}\right|_{\varepsilon = 0} \omega(\varepsilon)  = \left[\bar{\mathbf{J}}(\delta q), \bar{\mathbf{J}}(\dot{q})\right].\nonumber
  \end{equation}
\end{proof}
It follows from this and from Lemma \ref{DeltaJ_deltaq} that
\begin{align}
  \left(\left.\frac{\delta \bar{\mathbf{J}}}{\delta q}\right|_{(q, \dot{q})}\right)^*: \mu \mapsto \left(\left(\left[\bar{\mathbf{J}}(\dot{q}), \mu^\sharp\right]\right)_Q(q)\right)^\flat, \quad  \left(\left.\frac{\delta \bar{\mathbf{J}}}{\delta \dot{q}}\right|_{(q, \dot{q})}\right)^*: \mu \mapsto \left(\left(\mu^\sharp\right)_Q(q)\right)^\flat, 
\end{align}
for any $\mu \in \mathfrak{g}^*$. Taking the sharp of \eqref{Spline_Euler_Lagrange_equations} therefore gives
\begin{align}\label{EL_sym_sp_1}
  \left(\left[\bar{\mathbf{J}}, \ddot{\bar{\mathbf{J}}}\right]\right)_Q(q) - \frac{D}{Dt} \left[\left(\ddot{\bar{\mathbf{J}}}\right)_Q(q)\right] = 0.
\end{align}
It follows from \eqref{Horizontality_relations_1} and relations \eqref{Symmetric_space_inclusions_dash} that $\ddot{\bar{\mathbf{J}}} + \left[\dot{\bar{\mathbf{J}}}, \bar{\mathbf{J}}\right]$ is the horizontal generator of $\big(\ddot{\bar{\mathbf{J}}}\big)_Q(q)$ and therefore
\begin{align}
  & \frac{D}{Dt} \left[\left(\ddot{\bar{\mathbf{J}}}\right)_Q(q)\right] =  \frac{D}{Dt} \left[\left(\ddot{\bar{\mathbf{J}}} + \left[\dot{\bar{\mathbf{J}}}, \bar{\mathbf{J}}\right]\right)_Q(q)\right] \nonumber \\ &\qquad =  \left(\dddot{\bar{\mathbf{J}}} + \frac{3}{2}\left[\ddot{\bar{\mathbf{J}}}, \bar{\mathbf{J}}\right] - \frac{1}{2} \left[\bar{\mathbf{J}}, \left[\dot{\bar{\mathbf{J}}}, \bar{\mathbf{J}}\right]\right]\right)_Q(q)
= \left(\dddot{\bar{\mathbf{J}}} +\left[\ddot{\bar{\mathbf{J}}}, \bar{\mathbf{J}}\right]\right)_Q(q)\nonumber,
\end{align}
where in the last equality we used \eqref{Symmetric_space_inclusions_dash}, namely
\begin{equation} 
\big(\big[\bar{\mathbf{J}}, \big[\dot{\bar{\mathbf{J}}}, \bar{\mathbf{J}}\big]\big]\big)_Q(q) = -\big(\big[\bar{\mathbf{J}}, \operatorname{V}_q\big(\ddot{\bar{\mathbf{J}}}\big)\big]\big)_Q(q) =  -\big(\big[\bar{\mathbf{J}}, \ddot{\bar{\mathbf{J}}}\big]\big)_Q(q).
\end{equation}
Therefore, \eqref{EL_sym_sp_1} becomes \eqref{Cubic_symmetric_spaces_dash},
\begin{equation}\label{Cubic_symmetric_spaces}
  \operatorname{H}_q\left(\dddot{\bar{\mathbf{J}}} + 2\left[\ddot{\bar{\mathbf{J}}}, \bar{\mathbf{J}}\right]\right) = 0.
\end{equation}
We will exploit the similarity of this equation with the NHP equation \eqref{NHP_equation} when we analyze the horizontal lifts of cubics in the next section.
\begin{remark}{\rm The equations for Riemannian cubics in symmetric spaces were first given in \cite{CrSL1995}. We briefly remark on how those equations are related to \eqref{Cubic_symmetric_spaces}. We first note that \eqref{Horizontality_relations_2} together with \eqref{Symmetric_space_inclusions_dash} show that the vertical part of $\dddot{\bar{\mathbf{J}}} + 2\left[\ddot{\bar{\mathbf{J}}}, \bar{\mathbf{J}}\right]$ vanishes. Hence,
    \begin{equation}\label{Cubic_symmetric_spaces_augmented}
      \dddot{\bar{\mathbf{J}}} + 2\left[\ddot{\bar{\mathbf{J}}}, \bar{\mathbf{J}}\right] = 0.
\end{equation}
Let $g(t)$ be a horizontal lift of a Riemannian cubic $q(t)$, and define $\mathbf{V}(t) \in \mathfrak{g}_a^H$ by 
\begin{equation}
  \mathbf{V} = \operatorname{Ad}_{g^{-1}}\bar{\mathbf{J}}.
\end{equation}
One checks easily that \eqref{Cubic_symmetric_spaces_augmented} implies
\begin{equation}
  \dddot{\mathbf{V}} + \left[\mathbf{V}, \left[\dot{\mathbf{V}}, \mathbf{V}\right]\right] = 0.
\end{equation}
This coincides with equation (46) of \cite{CrSL1995}.
}
\end{remark}
\paragraph{Example: $G = SO(3)$ and $Q = S^2$.} Let $G = SO(3)$ and $Q = S^2 \subset \mathbb{R}^3$ and let $SO(3)$ act on $S^2$ through its action on vectors in $\mathbb{R}^3$. We start with our notational conventions.
\begin{remark}{\rm [Conventions for $SO(3)$ and $S^2$]\label{Remark_Conventions}\newline
Throughout this paper we use vector notation for the Lie algebra $\mathfrak{so}(3)$ of the Lie group of rotations $SO(3)$, as well as for its dual $\mathfrak{so}(3)^*$. One identifies 
$\mathfrak{so}(3)$ with $\mathbb{R}^3$ via the familiar 
isomorphism
\begin{equation}\label{hat_map}
\,\widehat{\,}:  \mathbb{R}^3\rightarrow \mathfrak{so}(3),\quad\mathbf{\Omega}=\left(
\begin{array}{c}
a\\
b\\
c\\
\end{array}\right) \mapsto \Omega:=\widehat{\mathbf{\Omega}}=\left(
\begin{array}{ccc}
0&-a&b\\
a&0&-c\\
-b&c&0
\end{array}
\right),
\end{equation}
called the \emph{hat map}.
This is a Lie algebra isomorphism when the vector cross product $\times$ is used as the Lie bracket operation on $\mathbb{R}^3$. The identification of $\mathfrak{so}(3)$ with $\mathbb{R}^3$ induces an isomorphism of the dual spaces $\mathfrak{so}(3)^* \cong \left(\mathbb{R}^3\right)^* \cong \mathbb{R}^3$. We represent tangent and cotangent spaces of $S^2$ as 
\begin{align}
T_{\mathbf{x}}S^2 = \left\{(\mathbf{x}, \mathbf{v})\in 
S^2\times\mathbb{R}^3 \mid \mathbf{x}\cdot \mathbf{v}=0\right\}, 
\quad  T_{\mathbf{x}}^*S^2=\left\{(\mathbf{x}, \mathbf{p}) \in 
S^2\times \mathbb{R}^3 \mid \mathbf{x}\cdot \mathbf{p} = 0
\right\}
\end{align}
with duality pairing $\left<\left(\mathbf{x}, \mathbf{p}\right), 
\left(\mathbf{x}, \mathbf{v}\right)\right>_{T^*S^2\times TS^2} 
= \mathbf{p} \cdot \mathbf{v}$.
Whenever admissible, we will drop the explicit mention of the 
base point $\mathbf{x}$ in what follows.}
\end{remark}
The infinitesimal generator is given by 
$(\mathbf{\Omega})_{S^2}(\mathbf{x}) 
= \mathbf{\Omega} \times \mathbf{x}$. We consider the bi-invariant extension $\gamma_{SO(3)}$ to 
$SO(3)$ of the identity moment of inertia inner product 
$\left<\mathbf{\Omega},\mathbf{\Omega}\right>_{\mathfrak{so}(3)} 
= \mathbf{\Omega} \cdot \mathbf{\Omega} = \|\mathbf{\Omega}\|^2$ 
on $\mathfrak{so}(3)$. The corresponding normal metric on $S^2$ 
is the round metric. The vertical and horizontal spaces are, 
see \eqref{Vertical_subspace_Lie_algebra_definition},
\begin{align}
    \mathfrak{so}(3)^V_{\mathbf{x}}&=\left\{ \mathbf{\Omega} \in \mathfrak{so}(3) \big| \mathbf{\Omega}= \lambda \mathbf{x} \mbox{ for } \lambda \in \mathbb{R}\right\} \quad \mbox{and} \quad 
 \mathfrak{so}(3)^H_{\mathbf{x}}=\left\{\mathbf{\Omega} \in \mathfrak{so}(3) \big| \mathbf{\Omega}\cdot \mathbf{x} = 0 \right\}.
\end{align}
The map $\bar{\mathbf{J}}$ of \eqref{Momentum_map_horizontal_generator_prop} becomes $\bar{\mathbf{J}}(\mathbf{x}, \mathbf{v}) = \mathbf{x} \times \mathbf{v}$. Equation \eqref{Cubic_symmetric_spaces} is
\begin{equation}\label{Spline_equation_sphere}
   \mathbf{x} \times (\dddot{\bar{\mathbf{J}}} + 2 \ddot{\bar{\mathbf{J}}}\times \bar{\mathbf{J}}) = 0, \quad \mbox{with}\quad  \bar{\mathbf{J}} = \mathbf{x}\times \dot{\mathbf{x}}.
\end{equation}
 Equation \eqref{Spline_equation_sphere} appears in \cite{Krakowski2005}, where it is derived from the general Euler-Lagrange equation for cubics \eqref{ELeqns-T1}. The similarity of \eqref{Spline_equation_sphere} with the NHP equation \eqref{NHP_equation} on $SO(3)$, $\dddot{\bar{\mathbf{J}}} + \ddot{\bar{\mathbf{J}}} \times \bar{\mathbf{J}} = 0$, was remarked upon there. We will take advantage of this similarity in Section \ref{Sec-Horizontal_lifts} for the investigation of horizontal lifts of cubics.
\subsection{Curvature from cubics}\label{Sec-Curvature_from_cubics}
One can infer the Christoffel symbols for a given metric if one knows the geodesic equation. It is interesting to note that in a similar way one obtains expressions for the (sectional) curvature from the equation of Riemannian cubics. We illustrate this in the case of Lie groups and symmetric spaces.
\paragraph{Lie groups.} Consider a Lie group $G$ with metric $\gamma_G$ that we assume, for simplicity, to be bi-invariant. Analogous arguments apply in the non-bi-invariant case. Let $g(t)$ be a curve in $G$ and $\bar{\mathbf{J}} =TR_{g^{-1}} \dot{g}$ its right-invariant velocity. Recall that $\operatorname{ad}^\dagger = -\operatorname{ad}$ and from \eqref{Covariant_derivative_group_prop}, \eqref{Covariant_derivative_vf_prop} that
\begin{align}
\!\!\!\!
  D_t\dot{g} =\left( \partial _t \bar{\mathbf{J}}\right) _G(g), \quad D_t^2\dot{g} = \big(\ddot{\bar{\mathbf{J} }} - \frac{1}{2}[\bar{\mathbf{J} }, \dot{\bar{\mathbf{J} }}]\big)_G(g) =: (\nu)_G(g), \quad D_t^3\dot{g} = \big(\dot{\nu} - \frac{1}{2}[\bar{\mathbf{J} }, \nu]\big)_G(g), \nonumber
\end{align}
where we defined $\nu := \ddot{\bar{\mathbf{J} }} - \frac{1}{2}[\bar{\mathbf{J}}, \dot{\bar{\mathbf{J} }}]$. The general Euler-Lagrange equation for cubics \eqref{ELeqns-T1} becomes
\begin{equation}\label{Lie_group_spline_with_curvature}
\left(\dot{\nu} - \frac{1}{2}[\bar{\mathbf{J} },\nu]\right)_G(g) 
+ R\left(\dot{\bar{\mathbf{J}}}_G(g),\bar{\mathbf{J}}_G(g)
\right)\left(\bar{\mathbf{J} }_G(g)\right) = 0.
\end{equation}
On the other hand, if $g(t)$ is a spline, then the NHP equation \eqref{NHP_equation} is satisfied, and therefore $\dot{\nu} = \frac{1}{2}\left[\bar{\mathbf{J}}, \ddot{\bar{\mathbf{J}}}\right]$. Plugging this into \eqref{Lie_group_spline_with_curvature}, yields
\begin{equation}
\frac{1}{4}\left(\left[\bar{\mathbf{J}}, 
\left[\bar{\mathbf{J}}, \dot{\bar{\mathbf{J}}}\right]\right]\right)_G(g) +  R\left(\dot{\bar{\mathbf{J} }}_G(g), 
\bar{\mathbf{J} }_G(g)\right)\left(\bar{\mathbf{J} }_G(g)\right) = 0. \nonumber
\end{equation}
We conclude that for any $g \in G$ and $\xi, \eta \in \mathfrak{g}$,
\begin{equation}\label{Lie_group_curvature}
  R\big(\eta_G(g), \xi_G(g)\big)\xi_G(g) = -\frac{1}{4}\big([\xi, [\xi, \eta]]\big)_G(g).
\end{equation}
\paragraph{Symmetric spaces.} For symmetric spaces one derives in a similar fashion that for a cubic $q(t)$ with horizontal generator $\bar{\mathbf{J}} = \bar{\mathbf{J}}(\dot{q})$,
\begin{equation}
  D_t^3 \dot{q} = \left(\dddot{\bar{\mathbf{J}}} + \left[\ddot{\bar{\mathbf{J}}}, \bar{\mathbf{J}}\right]\right)_Q(q). \nonumber
\end{equation}
It follows from \eqref{ELeqns-T1} and \eqref{Cubic_symmetric_spaces} that
\begin{equation}
  R(D_t\dot{q}, \dot{q}) \dot{q} = \left(\left[\ddot{\bar{\mathbf{J}}}, \bar{\mathbf{J}}\right]\right)_Q(q) = - \left(\left[\bar{\mathbf{J}}, \left[\bar{\mathbf{J}}, \dot{\bar{\mathbf{J}}}\right]\right]\right)_Q(q). \nonumber
\end{equation}
We conclude that for any $q \in Q$ and $\eta, \xi \in \mathfrak{g}_q^H$,
\begin{equation}\label{Symmetric_space_curvature}
  R \left(\eta_Q(q), \xi_Q(q)\right)\xi_Q(q) = - \left([\xi, [\xi, \eta]]\right)_Q(q). \nonumber
\end{equation}
\section{Horizontal lifts}\label{Sec-Horizontal_lifts}
The horizontal lifting property of geodesics in the normal metrics context has been an important feature of the large deformation matching framework in computational anatomy. References \cite{HoRaTrYo2004, MiTrYo2006, Younes2009S40}, amongst others, explore this aspect in detail. This motivates us to ask which Riemannian cubics on the manifold $Q$ lift to horizontal cubics on the Lie group $G$. In the case of symmetric spaces we give a complete characterization of the cubics that can be lifted horizontally. We then consider the more general setting of Riemannian submersions and formulate necessary and sufficient conditions under which horizontal lifts are possible.
\subsection{Symmetric spaces}\label{Subsec-Horizontal_lifts_symmetric_spaces}
Here we completely characterize the Riemannian cubics that can be lifted horizontally to the group of isometries. Let $(G, Q, \sigma)$ be a \emph{symmetric space structure}, as defined in Section \ref{Sec-Splines_on_symmetric_spaces}. In particular we recall the important relations \eqref{Symmetric_space_inclusions_dash},
\begin{equation}\label{Symmetric_space_inclusions}
  \left[\mathfrak{g}_q^V, \mathfrak{g}_q^V\right] \subset \mathfrak{g}_q^V, \quad  \left[\mathfrak{g}_q^H, \mathfrak{g}_q^V\right] \subset \mathfrak{g}_q^H, \quad  \left[\mathfrak{g}_q^H, \mathfrak{g}_q^H\right] \subset \mathfrak{g}_q^V.
\end{equation}
\begin{theorem}\label{Lifting_symmetric_spaces}
 A curve $q(t) \in Q$ is a Riemannian cubic and can be lifted horizontally to a Riemannian cubic $g(t) \in  G$ if and only if it satisfies $\dot{q}(t) = (\xi(t))_Q(q(t))$ for a curve $\xi(t) \in \mathfrak{g}$ of the form
 \begin{equation}\label{Quadratic_polynomial_symmetric_spaces}
   \xi(t) = \frac{ut^2}{2} + vt + w,
 \end{equation}
where $u, v, w$ span an Abelian subalgebra that lies in $\mathfrak{g}_{q(0)}^H$.
\end{theorem}
\begin{proof}
Suppose $q(t)$ is a cubic and can be lifted horizontally to a cubic $g(t)$. The horizontal generator curve $\bar{\mathbf{J}}(\dot{q}(t))$ simultaneously satisfies equations \eqref{NHP_equation} and \eqref{Cubic_symmetric_spaces_augmented}. Therefore,  $[\ddot{\bar{\mathbf{J}}}, \bar{\mathbf{J}}]= 0$. In particular
\begin{equation}
\gamma\left(\left[\ddot{\bar{\mathbf{J}}}, \bar{\mathbf{J}}\right], \dot{\bar{\mathbf{J}}}\right) =  \gamma\left(\ddot{\bar{\mathbf{J}}}, \left[\bar{\mathbf{J}}, \dot{\bar{\mathbf{J}}}\right]\right)  = \left\|\left[\bar{\mathbf{J}}, \dot{\bar{\mathbf{J}}}\right]\right\|^2_{\mathfrak{g}} = 0,\nonumber
  \end{equation}
where we used \eqref{Horizontality_relations_1} and \eqref{Symmetric_space_inclusions}. We conclude that  $[\bar{\mathbf{J}}, \dot{\bar{\mathbf{J}}}]= 0$.  Together with the NHP equation \eqref{NHP_equation_integrated_once} this reveals that $\ddot{\bar{\mathbf{J}}}$ is constant. Therefore,
  \begin{equation}\nonumber
    \bar{\mathbf{J}}(\dot{q}(t)) = \frac{ut^2}{2} +v t + w
  \end{equation}
with constants $u, v, w \in \mathfrak{g}$. These are mutually commuting because of $[\bar{\mathbf{J}}, \dot{\bar{\mathbf{J}}}]= 0$ and its time derivative $[\bar{\mathbf{J}}, \ddot{\bar{\mathbf{J}}}]= 0$. Moreover, $\mathbf{J}$ and $\dot{\mathbf{J}}$ are horizontal. Due to  $[\bar{\mathbf{J}}, \dot{\bar{\mathbf{J}}}]= 0$,  $\ddot{\mathbf{J}}$ is also horizontal. Therefore, $u, v, w \in \mathfrak{g}_{q(0)}^H$. Setting $\xi(t) = \bar{\mathbf{J}}(\dot{q}(t))$ therefore completes the first part of the proof. For the reverse, let $q(t)$ be a curve that satisfies  $\dot{q}(t) = (\xi(t))_Q(q(t))$ for $\xi(t) \in \mathfrak{g}$ of the form \eqref{Quadratic_polynomial_symmetric_spaces} with mutually commuting $u, v, w \in \mathfrak{g}_{q(0)}^H$. Fix an element $a \in Q$ and recall the Riemannian submersion $\Pi: G \rightarrow Q$ given by $g \mapsto ga$. Assume without loss of generality that $q(0) = a$. We first show that $\xi(t)$ is horizontal at all times. Then $\dot{\xi}(t)$ and $\ddot{\xi}(t)$ are also horizontal, since the commutators $\left[\xi, \dot{\xi}\right]$ and $ \left[\xi, \ddot{\xi}\right]$ vanish. Define the curve $g(t) \in G$ by $g(0) = e$ and $\dot{g} = TR_g \xi$. This curve lies in the Abelian subgroup $\operatorname{Exp}(\operatorname{span}(u, v, w))$ of $G$ with Lie algebra $\operatorname{span}(u, v, w) \subset \mathfrak{g}_{q(0)}^H$. Therefore, $ TL_{g^{-1}} \dot{g} = \operatorname{Ad}_{g^{-1}}\xi \in \operatorname{span}(u, v, w)$ which means that $\xi \in \operatorname{Ad}_{g}\operatorname{span}(u, v, w) \subset \operatorname{Ad}_g \mathfrak{g}_{q(0)}^H = \mathfrak{g}_{q(t)}^H$. Hence $\xi(t)$ is the horizontal generator of $q(t)$, and therefore $g(t)$ is a horizontal lift of $q(t)$. Moreover, $g(t)$ is a cubic in $G$ since $\xi(t)$ satisfies the NHP equation \eqref{NHP_equation}. It also satisfies \eqref{Cubic_symmetric_spaces}. Therefore, $q(t)$ is a cubic in $Q$. This concludes the proof.
\end{proof}
A \emph{Cartan subalgebra (CSA) based at} $q \in Q$ is a maximal Abelian Lie subalgebra contained in $\mathfrak{g}_q^H$. The \emph{rank} of the symmetric space is the dimension of its CSAs. The greater the rank of a symmetric space, the larger the set of vectors $u, v, w$ consistent with the requirements of Theorem \ref{Lifting_symmetric_spaces}.
\begin{corollary}\label{Corollary_rank_one}
  In rank-one symmetric spaces the only Riemannian cubics that can be lifted horizontally to Riemannian cubics on the group of isometries are geodesics composed with a cubic polynomial in time.
\end{corollary}
\begin{proof}
  Let $Q$ be a rank-one symmetric space. Then any CSA is one-dimensional. A curve $q(t)$ is therefore a cubic that can be lifted horizontally if and only if  $\dot{q}(t) = (\xi(t))_Q(q(t))$ with $\xi(t) \in \mathfrak{g}$ of the form
 \begin{equation}
   \xi(t) = \left(\frac{at^2}{2} + bt + c\right) d,
 \end{equation}
where $d \in \mathfrak{g}_{q(0)}^H$ and $a, b, c \in \mathbb{R}$. Therefore,
\[
  q(t) = e^{\left(\frac{at^3}{6} + \frac{bt^2}{2} + ct\right) d}q(0),
\]
which corresponds to the geodesic $y(t) = e^{td}q(0)$, composed with a cubic polynomial in time.
\end{proof}
\begin{remark}{\rm
Consider $G = SO(3)$ and the rank-one symmetric space $Q = S^2$.  The special cubics appearing in Corollary \ref{Corollary_rank_one} are considered in more detail in \cite{PaRa1997}. In particular, the case $b = c = 0$ corresponds to the so-called natural splines in computer aided design (CAD) applications.
}
\end{remark}
\subsection{Riemannian submersions} \label{Sec_Riemannian_submersion_viewpoint}
In this section we generalize the question of horizontal lifts of cubics to the Riemannian submersion setting. We also show how the result implies Theorem \ref{Lifting_symmetric_spaces} of the previous section. The result is the following: 
\begin{theorem}\label{RS_Lifting_theorem_statement}
    Let $\Pi: \tilde{Q} \rightarrow Q$ be a Riemannian submersion, and let $q(t) \in  Q$ be a Riemannian cubic. Moreover, let $\tilde{q}(t) \in \tilde{Q}$ be a horizontal lift of $q(t)$. The curve $\tilde{q}(t)$ is a Riemannian cubic if and only if
\begin{align}\label{RS_Lifting_proposition}
    \left[\dot{\tilde{q}}, \frac{D}{Dt}\dot{\tilde{q}}\right]^V = 0,\quad \text{ for all } t,
\end{align}
where the superscript $V$ denotes the vertical part.
  \end{theorem}
\begin{remark}{\rm
For any two horizontal vectors $\tilde{v}, \tilde{w} \in T^H_{\tilde{q}}\tilde{Q}$, the expression $[\tilde{v}, \tilde{w}]^V$ is defined as $[\tilde{v}, \tilde{w}]^V := [\tilde{V}, \tilde{W}]^V(\tilde{q})$, for horizontal extensions $\tilde{V}, \tilde{W}$ of $\tilde{v}, \tilde{w}$.
}
\end{remark}
\begin{proof}
We denote $v(t) := \dot{q}(t)$ and $\tilde{v}(t) := 
\dot{\tilde{q}}(t)$. The metrics on $Q$ and $\tilde{Q}$ are denoted $\tilde{\gamma}$ and $\gamma$. The covariant derivatives with respect to the Levi-Civita connections are written 
$\nabla$ and $\tilde{\nabla}$ respectively. Recall that by definition $\frac{D}{Dt} v = \nabla_{v}v$ and $\frac{D}{Dt}\tilde{v} = \tilde{\nabla}_{\tilde{v}}\tilde{v}$. Using the formula $\tilde{\nabla}_{\tilde{X}}\tilde{Y} = \widetilde{\nabla_XY} + \frac{1}{2}[\tilde{X}, \tilde{Y}]^V$ for the covariant derivative induced by a Riemannian submersion of horizontally lifted vector fields $\tilde{X}$ and $\tilde{Y}$, one obtains
\begin{align}
  &\tilde{\nabla}_{\tilde{v}}\tilde{v} = \widetilde{\nabla_v v}, \quad (\tilde{\nabla}_{\tilde{v}})^2 \tilde{v} = \widetilde{(\nabla_{v})^2 v} + \frac{1}{2} [\tilde{v},\widetilde{\nabla_{v} v}]^V, \label{Intermediate_2}
\\
& (\tilde{\nabla}_{\tilde{v}})^3 \tilde{v} = \widetilde{(\nabla_{v})^3 v} + 
\frac{1}{2} [\tilde{v},\widetilde{(\nabla_{v}^2) v}]^V + 
\frac{1}{2} \tilde{\nabla}_{\tilde{v}}  ([\tilde{v},\widetilde{\nabla_{v} v}]^V).\label{Intermediate_3}
\end{align}
Suppose $q(t)$ and $\tilde{q}(t)$ are as in the statement of the 
theorem and let $\tilde{q}(t)$ be a Riemannian cubic spline, 
i.e., the respective spline equations are satisfied,
\begin{align}
  \begin{cases}
    (\nabla_{v})^3 v + R(\nabla_v v,v) v = 0,\\ \label{System_of_Splines}
    (\tilde{\nabla}_{\tilde{v}})^3\tilde{v} + \tilde{R}(\tilde{\nabla}_{\tilde{v}} \tilde{v}, \tilde{v}) \tilde{v} = 0.
  \end{cases}
\end{align}
For the following manipulations we record that
\begin{align}\label{Partial_Integration}
& \tilde{\gamma}\left( \tilde{\nabla}_{\tilde{v}}  ([\tilde{v},\widetilde{\nabla_{v} v}]^V) ,  \tilde{\nabla}_{\tilde{v}}\tilde{v} \right)  =\frac{d}{dt}\tilde{\gamma}\left( [\tilde{v},\widetilde{\nabla_{v} v}]^V ,  \tilde{\nabla}_{\tilde{v}}\tilde{v}\right) -\tilde{\gamma}\left(  [\tilde{v},\widetilde{\nabla_{v} v}]^V,\tilde{\nabla}^2_{\tilde{v}} \tilde{v} \right)=  -\tilde{\gamma}\left( [\tilde{v},\widetilde{\nabla_{v} v}]^V,\tilde{\nabla}^2_{\tilde{v}} \tilde{v} \right), \nonumber
\end{align}
where the second step follows since $ \tilde{\nabla}_{\tilde{v}}\tilde{v}$ is horizontal. We use this equality as well as equations \eqref{Intermediate_2} - \eqref{Intermediate_3} to obtain 
\begin{align*}
  \tilde{\gamma}\left( (\tilde{\nabla}_{\tilde{v}})^3\tilde{v}, \tilde{\nabla}_{\tilde{v}}\tilde{v}\right) &= \tilde{\gamma}\left(\widetilde{(\nabla_{v})^3 v}, \widetilde{\nabla_v v} \right) + \frac{1}{2} \tilde{\gamma}\left( \tilde{\nabla}_{\tilde{v}}  ([\tilde{v},\widetilde{\nabla_{v} v}]^V),  \widetilde{\nabla_v v} \right)\\
&=  \tilde{\gamma}\left( \widetilde{(\nabla_{v})^3 v}, \widetilde{\nabla_v v} \right) - \frac{1}{2}  \tilde{\gamma}\left( [\tilde{v},\widetilde{\nabla_{v} v}]^V, (\tilde{\nabla}_{\tilde{v}}  )^2 \tilde{v}\right) \\
&=  \tilde{\gamma}\left( \widetilde{(\nabla_{v})^3 v}, \widetilde{\nabla_v v} \right) - \frac{1}{4} \left\| [\tilde{v},\widetilde{\nabla_{v} v}]^V\right\|^2.
\end{align*}
Hence, 
\begin{equation}\label{Pairing_Cov_Ders}
   \tilde{\gamma}\left( (\tilde{\nabla}_{\tilde{v}})^3\tilde{v}, \tilde{\nabla}_{\tilde{v}}\tilde{v}\right) =  \gamma\left((\nabla_{v})^3 v, \nabla_v v \right) - \frac{1}{4} \left\| [\tilde{v},\widetilde{\nabla_{v} v}]^V\right\|^2.
\end{equation}
On the other hand O'Neill's formula for sectional curvatures of Riemannian submersions \cite{ONeill1966} (Equation $3.$ in Corollary 1), implies that
\begin{align}\label{Pairing_Cov_Der_Curvature}
 \tilde{\gamma}\left( \tilde{R}(\tilde{\nabla}_{\tilde{v}} \tilde{v}, \tilde{v}) \tilde{v}, \tilde{\nabla}_{\tilde{v}}\tilde{v}\right) = \gamma\left( R(\nabla_v v, v) v, \nabla_v v\right) - \frac{3}{4} \left\|[\tilde{v}, \widetilde{\nabla_v v}]^V\right\|^2.
\end{align}
Adding \eqref{Pairing_Cov_Ders} and \eqref{Pairing_Cov_Der_Curvature} and using the spline equations \eqref{System_of_Splines} we conclude that
\begin{equation}\label{First_condition_in_proof}
   [\tilde{v}, \widetilde{\nabla_v v}]^V = [\tilde{v}, \tilde{\nabla}_{\tilde{v}} \tilde{v}]^V = 0\,,
\end{equation}
which is \eqref{RS_Lifting_proposition}. To show the reverse 
direction, we note that when \eqref{RS_Lifting_proposition} holds, then \eqref{Intermediate_2} and \eqref{Intermediate_3} take the simplified form \begin{align}
   (\tilde{\nabla}_{\tilde{v}})^2 \tilde{v} = \widetilde{(\nabla_v)^2v}, \quad (\tilde{\nabla}_{\tilde{v}})^3 \tilde{v} = \widetilde{(\nabla_{v})^3 v} + 
\frac{1}{2} [\tilde{v},(\tilde{\nabla}_{\tilde{v}})^2 \tilde{v}]^V.\label{Intermediate_3_updated}
\end{align}
Hence, $(\tilde{\nabla}_{\tilde{v}})^2 \tilde{v}$ is horizontal. Moreover, $ (\tilde{\nabla}_{\tilde{v}})^3 \tilde{v}$ splits naturally into horizontal and vertical parts. Therefore, checking that the second equation in  \eqref{System_of_Splines} holds, amounts to verifying that for any choice of horizontal vector field $\tilde{h}(t)$ and any choice of vertical vector field $\tilde{w}(t)$ along $\tilde{q}(t)$,
\begin{align}
  \begin{cases}\label{Splitting_equations_horizontal_vertical}
    \tilde{\gamma}\left(\tilde{R}(\tilde{\nabla}_{\tilde{v}}\tilde{v}, \tilde{v}) \tilde{v}, \tilde{h}\right) - \gamma\left( R(\nabla_v v, v)v, h\right) = 0\\
\tilde{\gamma}\left( \tilde{R}(\tilde{\nabla}_{\tilde{v}}\tilde{v}, \tilde{v}) \tilde{v}, \tilde{w}\right) +\frac{1}{2} \tilde{\gamma}\left( [\tilde{v}, (\tilde{\nabla}_{\tilde{v}})^2\tilde{v}]^V, \tilde{w}\right) = 0,
  \end{cases}
\end{align}
where we denoted $h := \Pi_*\tilde{h}$. To proceed, we introduce the $(1,2)$-tensors $A$ and $T$ defined, for arbitrary vector fields $E, F$, by
 \begin{align}
   A_EF = \left(\tilde{\nabla}_{E^H}(F^H)\right)^V + 
\left(\tilde{\nabla}_{E^H}(F^V)\right)^H \label{Definition_A}\\
 T_EF =\left(\tilde{\nabla}_{E^V}(F^V)\right)^H + 
 \left(\tilde{\nabla}_{E^V}(F^H)\right)^V. \label{Definition_T}
 \end{align}
 The superscripts $H$ and $V$ denote the horizontal and vertical parts, respectively. Definitions \eqref{Definition_A} and \eqref{Definition_T} coincide with the ones given in 
 \cite{ONeill1966}. It is shown there (in Equations \{3\} and \{4\}) that if $X, Y, Z, H$ are horizontal vector fields and $W$ is a vertical vector field, then
\begin{align}\label{ONeill_4}
  \tilde{\gamma}\left( \tilde{R}(X, Y) Z, H\right) =&\mbox{ } \gamma\left( R(\Pi_*X, \Pi_*Y) \Pi_*Z, \Pi_*H\right) + 2\tilde{\gamma} \left( A_XY, A_ZH\right) \nonumber \\&- \tilde{\gamma}\left( A_YZ, A_XH\right) - \tilde{\gamma}\left( A_ZX, A_YH\right).
\end{align}
and
 \begin{align} \label{ONeill_3}
   \tilde{\gamma}\left(\tilde{R}(X, Y) Z, W \right) = &- \tilde{\gamma}\left((\tilde{\nabla}_ZA)_XY, W\right) - \tilde{\gamma}\left( A_XY, T_WZ\right) \nonumber \\&+ \tilde{\gamma}\left( A_YZ, T_WX\right) + \tilde{\gamma}\left( A_ZX, T_WY\right).
 \end{align}
Note that we differ from \cite{ONeill1966} in our sign convention for the curvature tensor. It is also shown in \cite{ONeill1966} that for any two horizontal vector fields $X$ and $Y$ one has $ A_XY = \frac{1}{2} [X, Y]^V$. In particular, \eqref{First_condition_in_proof} can be written as $A_{\tilde{v}} (\tilde{\nabla}_{\tilde{v}}\tilde{v}) = 0$. This, together with \eqref{ONeill_4}, implies that
 \begin{equation}
   \tilde{\gamma}\left( \tilde{R}(\tilde{\nabla}_{\tilde{v}}\tilde{v}, \tilde{v})\tilde{v}, \tilde{h}\right) = \gamma\left( R(\nabla_vv, v) v, h\right), \nonumber
 \end{equation}
which is equivalent to the first equation in 
\eqref{Splitting_equations_horizontal_vertical}. In order to show the second equation we take a covariant derivative of \eqref{First_condition_in_proof} written in the form 
$A_{\tilde{v}} (\tilde{\nabla}_{\tilde{v}}\tilde{v}) = 0$ 
to obtain
\begin{equation}
   0 = \tilde{\nabla}_{\tilde{v}}(A_{\tilde{v}}(\tilde{\nabla}_{\tilde{v}}\tilde{v})) = (\tilde{\nabla}_{\tilde{v}}A)_{\tilde{v}}(\tilde{\nabla}_{\tilde{v}}\tilde{v}) + A_{\tilde{v}}((\tilde{\nabla}_{\tilde{v}})^2\tilde{v}). \nonumber
\end{equation}
It follows from this and \eqref{ONeill_3} that
\begin{align}
  \tilde{\gamma}\left( \tilde{R}(\tilde{\nabla}_{\tilde{v}}\tilde{v}, \tilde{v}) \tilde{v}, \tilde{w}\right) &= - \tilde{\gamma}\left( \tilde{R}(\tilde{v}, \tilde{\nabla}_{\tilde{v}}\tilde{v}) \tilde{v}, \tilde{w}\right)
= \tilde{\gamma}\left( (\tilde{\nabla}_{\tilde{v}}A)_{\tilde{v}} (\tilde{\nabla}_{\tilde{v}}\tilde{v}), \tilde{w}\right) \nonumber  \\
&= - \tilde{\gamma}\left( A_{\tilde{v}}((\tilde{\nabla}_{\tilde{v}})^2\tilde{v}), \tilde{w}\right)
= -\frac{1}{2}\tilde{\gamma}\left( [\tilde{v}, (\tilde{\nabla}_{\tilde{v}})^2\tilde{v}]^V, \tilde{w}\right). \nonumber
\end{align}
Therefore the second equation of \eqref{Splitting_equations_horizontal_vertical} is satisfied. This concludes the proof.
\end{proof}
\subsubsection{Example: Normal metrics in the bi-invariant case}
We now show how this result relates to Theorem \ref{Lifting_symmetric_spaces} of Section \ref{Subsec-Horizontal_lifts_symmetric_spaces}. Let $G$ be a Lie group with a bi-invariant metric $\gamma$ that acts transitively on a manifold $Q$ equipped with the normal metric $\gamma_Q$. Recall that for a fixed element $a \in Q$ the map $\Pi: G \rightarrow Q, g \mapsto ga$ is a Riemannian submersion.
\begin{lemma}\label{A_lemma}
   Let $g \in G$ with $\Pi(g) = q$ and let $\xi$ and $\eta$ be in $\mathfrak{g}_q^H$. Then
   \begin{equation}\label{A_lemma_eqn}
     \left[\xi_G(g), \eta_G(g)\right]^V = \left(\operatorname{V}_q\left([\xi, \eta]\right)\right)_G(g).
   \end{equation}
\end{lemma}
\begin{proof}
 In order to compute the left hand side we extend $\xi_G(g)$ to a horizontal vector field $H^\xi$ on $G$. Namely, $H^\xi(h) = l_{\operatorname{Ad}_{g^{-1}}(\xi)}(h)$, where $l_\eta$ denotes the left-invariant vector field $l_\nu(h) = h \nu$ for any $\nu \in \mathfrak{g}$. Similarly $H^\eta(h) =  l_{\operatorname{Ad}_{g^{-1}}(\eta)}(h)$. Now
\begin{equation}
    \left[\xi_G(g), \eta_G(g)\right]^V  = \left[H^\xi, H^\eta\right]^V(g) = \left(l_{\operatorname{Ad}_{g^{-1}}[\xi, \eta]}(g)\right)^V= \left([\xi, \eta]_G(g)\right)^V = \left(\operatorname{V}_q\left([\xi, \eta]\right)\right)_G(g).\nonumber
\end{equation}
\end{proof}
\begin{theorem}\label{Normal_metrics_lifting_proposition_statement}
   Let $q(t) \in Q$ be a Riemannian cubic spline with horizontal generator $\bar{\mathbf{J}}(\dot q(t)) \in \mathfrak{g}_{q(t)}^H$. A horizontal lift $g(t) \in G$ of $q(t)$ is a Riemannian cubic if and only if
\begin{align}\label{Normal_metrics_lifting_proposition}
    \operatorname{V}_q\left([\bar{\mathbf{J}}, \dot{\bar{\mathbf{J}}}]\right) = 0,\quad \text{ for all }t.
\end{align}
\end{theorem}
\begin{proof}
Since $\dot{g} = \left(\bar{\mathbf{J}}\right)_G(g)$ and $D_t\dot{g} =  \big(\dot{\bar{\mathbf{J}}}\big)_G(g)$ with both $\bar{\mathbf{J}}$ and $\dot{\bar{\mathbf{J}}}$ in $\mathfrak{g}_q^H$, Lemma \ref{A_lemma} gives
   \begin{equation}
     \left[\dot{g}, \frac{D}{Dt}\dot{g}\right]^V = \left(\operatorname{V}_q\left([\bar{\mathbf{J}}, \dot{\bar{\mathbf{J}}}]\right)\right)_G(g).
   \end{equation}
This expression vanishes according to Proposition \ref{RS_Lifting_proposition}, and therefore $\operatorname{V}_q\left([\bar{\mathbf{J}}, \dot{\bar{\mathbf{J}}}]\right) = 0$.
\end{proof}
\begin{remark}{\rm
If $Q$ is a symmetric space, then relations \eqref{Symmetric_space_inclusions} hold. Therefore $\operatorname{V}_q\left([\bar{\mathbf{J}}, \dot{\bar{\mathbf{J}}}]\right) = 0$ is equivalent to $ [\bar{\mathbf{J}}, \dot{\bar{\mathbf{J}}}] = 0$. This property was the basis for the proof of Theorem \ref{Lifting_symmetric_spaces}.
}
\end{remark}
\section{Extended analysis: Reduction by isotropy subgroup}\label{Section-Extended_analysis_Reduction_by_isotropy_subgroup}
In the previous section, we analyzed the relationship between Riemannian cubics on $Q$ and horizontal curves on $G$. More precisely, we gave a necessary and sufficient condition 
guaranteeing the existence of a horizontally lifted cubic on 
$G$ covering a given cubic on $Q$. In the present section we include in our considerations the non-horizontal curves on $G$. We show that certain non-horizontal geodesics on $G$ project to cubics. We then extend the analysis in the following way. We reduce the Riemannian cubic variational problem on $G$ by the isotropy subgroup $G_a$ of a point $a \in Q$. The reduced Lagrangian couples horizontal and vertical parts of the motion, which accounts for the absence of a general horizontal lifting property. The reduced form of the equations reveals the obstruction for a cubic on $G$ to project to a cubic on $Q$. The main technical tool in this section is second-order Lagrange--Poincar\'e reduction. The main references are \cite{CeMaRa2001} for the first-order theory and \cite{Gay-BalmazEtAl2011HOLPHP} for the recent generalization to higher order.
\subsection{Setting}
Let $G$ be a Lie group with bi-invariant metric $\gamma_G$, acting transitively from the left on a manifold $Q$ with the normal metric $\gamma_Q$. Recall the Riemannian submersion $\Pi: G \rightarrow Q$ given by $g \mapsto ga$ for a fixed element $a \in Q$. Let $G_a \subset G$ be the stabilizer of $a$. Consider the right action
\[
    \psi: G \times G_a \rightarrow G\,, \quad (g, h) \mapsto gh\,.
\]
This action is free and the projection from $G$ onto the quotient space $G/G_a$ is a submersion. Moreover the map $G/G_a \rightarrow Q$ given by $[g] \mapsto ga$ is a diffeomorphism. The ingredients $\left(G, Q \cong G/G_a, G_a, \Pi, \psi\right)$ therefore constitute a principal fiber bundle with total space $G$, base manifold $Q$, structure group $G_a$, projection $\Pi$, and action $\psi$. The Lie algebra of the structure group $G_a$ is $\mathfrak{g}_a= \mathfrak{g}_a^V$. Recall the vertical and horizontal projections, 
$\operatorname{V}_q: \mathfrak{g} \rightarrow \mathfrak{g}_q^V$ and $\operatorname{H}_q: \mathfrak{g} \rightarrow 
\mathfrak{g}_q^H$, for any $q \in Q$. Note that for any 
$g \in G$ with $ga = q$,
  \begin{align}\label{Ad_relations_projections}
     \operatorname{Ad}_g \circ \operatorname{H}_a \circ \operatorname{Ad}_g^{-1} = \operatorname{H}_q \quad \mbox{and} \quad \operatorname{Ad}_g \circ \operatorname{V}_a \circ \operatorname{Ad}_g^{-1} = \operatorname{V}_q\,.
  \end{align}
We equip $G$ with the principal bundle connection $\mathcal{A}$, 
\begin{equation}\label{PFB_mechanical_connection}
   \mathcal{A}: TG \rightarrow \mathfrak{g}_a\,, \quad v_g \mapsto \mathcal{A}_g(v_g):= \operatorname{V_a}(TL_{g^{-1}} v_g).
\end{equation}
Recall from Section \ref{Section_Normal_metrics} that $\gamma_G$, together with the projection $\Pi$, induces a splitting of $TG$ into horizontal and vertical subbundles $TG = T^HG \oplus T^VG$. This splitting coincides with the one prescribed by the connection $\mathcal{A}$, that is, $T_g^HG = \ker \mathcal{A}_g$ and $T_g^VG = \ker T_g\Pi$. The curvature of $\mathcal{A}$ is the $\mathfrak{g}$-valued two-form 
\begin{align}
\begin{split}
\label{LP_curvature}
\mathcal{B} (u _g , v _g )&=[ V_a( TL_{g ^{-1}} u _g ),V_a( TL_{g ^{-1}} v _g ) ]- V_a \left( [TL_{g ^{-1}} u _g , TL_{g ^{-1}} v _g ] \right) \\
&= - V_a \left( [H_a (TL_{g ^{-1}} u _g ),H_a( TL_{g ^{-1}} v _g ) ]\right) 
\end{split}
\end{align}
for $u_g, v_g \in T_gG$. Consider the following action of $G_a$ on $G \times \mathfrak{g}_a$,
\begin{equation}
  G_a \times (G \times \mathfrak{g}_a) \rightarrow G \times \mathfrak{g}_a, \quad (h, g, \xi) \mapsto (gh, \operatorname{Ad}^{-1}_h \xi).
\end{equation}
We define the associated adjoint vector bundle over $Q$, $\tilde{\mathfrak{g}}_a := (G \times \mathfrak{g}_a)/G_a$. The
equivalence class, i.e., orbit, of $(g, \xi) \in G \times \mathfrak{g}_a$ will be denoted by square brackets, 
$\sigma =[g, \xi] \in \tilde{\mathfrak{g}}_a$. The principal connection $\mathcal{A}$ induces a linear connection on  $\tilde{\mathfrak{g}}_a$ with covariant derivative
\begin{equation}\label{Cov_der_adjoint_bundle}
  \frac{D^{\mathcal{A}}}{Dt} [g(t), \xi(t)] = \left[g(t), \dot{\xi}(t) + \left[\mathcal{A}(\dot{g}(t)), \xi(t)\right]\right].
\end{equation}
 We will sometimes use the shorthand $\dot{\sigma} := \frac{D^{\mathcal{A}}}{Dt} \sigma$ where $\sigma(t)$ is a curve in $\tilde{\mathfrak{g}}_a$. Moreover, we define the map
\begin{equation}\label{Ad_map_PFB}
i:\tilde{\mathfrak{g}}_a \rightarrow \mathfrak{g}, \quad 
[g, \eta] \mapsto i([g, \eta ]):=\operatorname{Ad}_{g} \eta,
\end{equation}
and write $[g, \eta]=:\sigma \mapsto \bar \sigma :=i([g,\eta])$, 
as shorthand. Note that $i([g, \eta ]) \in \mathfrak{g}_{ga}^V$. We introduce the fiber-wise inner product $\bar{\gamma}$ on $\tilde{\mathfrak{g}}_a$ given by
\[
\bar{\gamma}(\sigma, \rho) := \gamma(\bar{\sigma}, \bar{\rho})
\]
and its corresponding norm is denoted by $\|.\|_{\tilde{\mathfrak{g}}_a}$.
We define the $\tilde{\mathfrak{g}}_a$-valued reduced curvature $2$-form $\widetilde{\mathcal{B}}$,
\[
\widetilde{\mathcal{B}}(u_q, v_q) := [g, \mathcal{B}(u_g, v_g)],
\]
for $u_q, v_q \in T_qQ$, where $g \in G$ and $u_g, v_g \in T_gG$ are such that $\Pi(g) = q$ and $T_g \Pi(u_g) = u_q$, $T_g \Pi(v_g) = v_q$. 
\subsection{First-order Lagrange--Poincar\'e reduction}
We start by recalling first-order Lagrange--Poincar\'e reduction, which makes use of the bundle diffeomorphism
\begin{equation}\label{Bundle_diffeomorphism_alpha_1}
  \alpha_{\mathcal{A}}^{(1)}: TG/G_a \rightarrow TQ \times_Q \tilde{\mathfrak{g}}_a, \quad [g, \dot{g}] \mapsto (q, \dot{q}) \times [g, \mathcal{A}(\dot{g})].
\end{equation}
Here we introduced the quotient $TG/G_a$ of $TG$ by the natural action of $G_a$, whose elements we denote by $[g, \dot{g}] \in TG/G_a$. We also defined $(q, \dot{q}) := T_g \Pi(g, \dot{g})$ for any representative $(g,\dot{g})$ of $[g, \dot{g}]$. Let $L: TG \rightarrow \mathbb{R}$ be a $G_a$-invariant Lagrangian. The reduced Lagrangian $\ell$ is defined by $L = \ell \circ \alpha^{(1)}_{\mathcal{A}}$. In order to compute the Lagrange--Poincar\'e equations one takes constrained variations in the reduced variable space. Take for instance the kinetic energy Lagrangian $L(g, \dot{g}) = \frac{1}{2} \left\|\dot{g}\right\|_g^2$. Hamilton's principle $\delta S = 0$ for $S = \int_0^1 L \, dt$ leads to the Euler-Lagrange equation $D_t \dot{g} = 0$. This is the geodesic equation on $G$. We derive the corresponding Lagrange--Poincar\'e equations. The reduced Lagrangian $\ell: TQ \times_Q \tilde{\mathfrak{g}}_a$ is given by
\begin{equation}
  \ell(q, \dot{q}, \sigma) = \frac{1}{2} \left\|\dot{q}\right\|^2_q + \frac{1}{2} \left\|\sigma\right\|^2_{\tilde{\mathfrak{g}}_a} =  \frac{1}{2} \left\|\dot{q}\right\|^2_q + \frac{1}{2} \left\|\bar{\sigma}\right\|^2_{\mathfrak{g}}, 
\end{equation}
where we recall that $\bar{\sigma} := i(\sigma)$. We will need the following result.
\begin{lemma}\label{Lemma_delta_sigma_bar} Consider the map $i:\tilde{\mathfrak{g}  } _a \rightarrow \mathfrak{g}  $ defined in \eqref{Ad_map_PFB} and a curve $ \sigma _\varepsilon \in \tilde{\mathfrak{g}  } _a$, covering the curve $ q _\varepsilon \in Q$. Then, we have the formula
\[
\frac{d}{d\varepsilon} i ( \sigma _\varepsilon )=i \left( \frac{D^ \mathcal{A} }{D\varepsilon} \sigma _\varepsilon \right) + F\left( \frac{d}{d\varepsilon} q _\varepsilon , \sigma \right),
\]
where the covariant derivative $D^{\mathcal{A}}$ was defined in \eqref{Cov_der_adjoint_bundle} and the map $F:TQ \times_Q  \tilde{ \mathfrak{g}  }_a \rightarrow \mathfrak{g}  $ is defined by
\[
F( v _q, \sigma _q ):= H_q\left( [ TR_{g^{-1}}v_g, \operatorname{Ad}_g \eta ] \right)  = [H_q(TR_{g^{-1}}v_g ), \operatorname{Ad}_g \eta ],
\]
where $ v _g \in TG$ and $ \eta \in \mathfrak{g}  _a $ are such that $T_g \pi \left( v _g \right) = v _q $ and $[g, \eta ]= \sigma _q $.
\end{lemma} 
In short, $\delta \bar{\sigma} = i(\delta \sigma) + F(\delta q, \sigma)$, where $\delta \sigma := \left.\frac{D^{\mathcal{A}}}{D\varepsilon}\right|_{\varepsilon = 0} \sigma$. Note that $ i(\delta \sigma) \in \mathfrak{g}_q^V$ and $F(\delta q, \sigma) \in \mathfrak{g}_q^H$. One computes
\begin{align}
  \delta S = \int_0^1 \gamma_Q( \dot{q}, D_\varepsilon \dot{q}_\varepsilon) + \gamma( \bar{\sigma}, \delta \bar{\sigma})\, dt =  \int_0^1 -\gamma_Q(D_t \dot{q}, \delta q) + \bar{\gamma}(\sigma, \delta \sigma) \, dt. \nonumber
\end{align}
Using the constrained variations
\[
\delta \sigma =\frac{D^ \mathcal{A} }{Dt} \eta -[ \sigma , \eta ]+\widetilde{\mathcal{B}}( \delta q, \dot q)\in \tilde{ \mathfrak{g}  }_a
\]
gives the horizontal and vertical Lagrange--Poincar\'e equations
\begin{align}\label{Original_LP_equations}
  \frac{D}{Dt} \dot{q} = - \left<\sigma, \mathbf{i}_{\dot{q}}\widetilde{\mathcal{B}}\right>^\sharp, \quad \frac{D^{\mathcal{A}}}{Dt} \sigma = 0,
\end{align}
where we defined
\[
  \gamma_Q\bigg(\left<\rho, \mathbf{i}_{v_q}\widetilde{\mathcal{B}}\right>^\sharp, w_q\bigg) = \bar{\gamma}\left(\rho, \mathbf{i}_{v_q}\widetilde{\mathcal{B}}(w_q)\right) = \bar{\gamma}\left(\rho, \widetilde{\mathcal{B}}(v_q, w_q)\right),
\]
for all $v_q, w_q \in T_qQ$ and $\rho \in \left(\tilde{\mathfrak{g}}_a\right)_q$.
Using the expression \eqref{LP_curvature} one computes that
\begin{equation}\label{B_tilde}
   \left<\rho, \mathbf{i}_{v_q}\widetilde{\mathcal{B}}\right>^\sharp = \left(\left[\bar{\mathbf{J}}(v_q), \bar{\rho}\right]\right)_Q(q) = - \nabla_{v_q}\bar{\rho}_Q,
\end{equation}
so that \eqref{Original_LP_equations} becomes
\begin{equation}\label{Geodesic_LP_equations}
   \frac{D}{Dt} \dot{q} =  \nabla_{\dot{q}}\bar{\sigma}_Q,   \quad  \frac{D^{\mathcal{A}}}{Dt} \sigma = 0.
\end{equation}
The first of these equations is the geodesic equation on $Q$ up to a forcing term on the right hand side. One recognizes again that horizontal geodesics ($\sigma = 0$) on $G$ project to geodesics on $Q$. 

\begin{remark}\rm \label{Rem_Ballistic_curves}The projections to $Q$ of non-horizontal geodesics in $G$ are called \emph{ballistic curves}  in \cite{AlKrLoMi2003}.
\end{remark} 
\begin{remark}\rm \label{EP_constant}
Note that since the Lagrangian $L(g, \dot{g}) = \frac{1}{2} \left\|\dot{g}\right\|_g^2$ is not only $G_a $-invariant but $G$-invariant, the Lagrange--Poincar\'e equations \eqref{Geodesic_LP_equations} can be further reduced to yield the Euler--Poincar\'e equations
\[
\frac{d}{dt} \xi = - \operatorname{ad}^\dagger _\xi \xi  =0, \quad \xi = TR_{g^{-1}}\dot g .
\]
We thus obtain that $ \xi $ is a constant. This form of the equation is however not adapted for our purpose since it does not involve the manifold $Q$.
\end{remark}
\begin{remark}\rm \label{Remark_LP_symmetric}
 If $Q$ is a symmetric space, then the first equation in \eqref{Geodesic_LP_equations} can be used together with the second equality of \eqref{B_tilde} to write
  \begin{equation}
    \dot{\bar{\mathbf{J}}}_Q(q) = \frac{D}{Dt} \dot{q} = - \left(\left[\bar{\mathbf{J}}, \bar{\sigma}\right]\right)_{Q}(q).  \nonumber
\end{equation}
Since both $\dot{\bar{\mathbf{J}}}$ and $\left[\bar{\mathbf{J}}, \bar{\sigma}\right]$ are in $\mathfrak{g}_{q}^H$ we conclude that $\dot{\bar{\mathbf{J}}} = - \left[\bar{\mathbf{J}}, \bar{\sigma}\right]$. The second equation in \eqref{Geodesic_LP_equations} together with Lemma \ref{Lemma_delta_sigma_bar} yields
\begin{align}
  \dot{\bar{\sigma}} = \frac{d}{dt} i(\sigma) = F(\dot{q}, \sigma) = \left[\bar{\mathbf{J}}, \bar{\sigma}\right], \nonumber
\end{align}
where we recall that $\bar{\sigma} := i(\sigma)$. For symmetric spaces, \eqref{Geodesic_LP_equations} is therefore equivalent to
\begin{align}
  \dot{\bar{\mathbf{J}}} = \left[\bar{\sigma}, \bar{\mathbf{J}}\right], \qquad \dot{\bar{\sigma}} =  \left[\bar{\mathbf{J}}, \bar{\sigma}\right].
\end{align}
In particular  $\xi = \bar{\mathbf{J}} + \bar{\sigma}$ is a constant, as in Remark \ref{EP_constant} above.
\end{remark}
\subsection{Example: $G = SO(3)$, $Q = S^2$.}
We work with the conventions defined in Remark \ref{Remark_Conventions}. Choose as anchor point $a$ the North pole $\mathbf{e}_z \in S^2$ and define the map
\begin{equation}
  (\cdot )_3 : \mathfrak{so}(3) \rightarrow \mathbb{R}, \quad \widehat{\mathbf{\Omega}} \mapsto (\widehat{\mathbf{\Omega}})_3:= \mathbf{\Omega} \cdot \mathbf{e}_z.\nonumber 
\end{equation}
The isotropy subgroup $SO(3)_a \cong S^1 \subset SO(3)$ corresponds to rotations around $\mathbf{e}_z$, and $\mathfrak{so}(3)_a$ is identified with $\mathbb{R}$ via the map $\lambda \mathbf{e}_z \mapsto \lambda$. The vector bundle $\widetilde{\mathfrak{so}(3)}_a$ is isomorphic to $S^2 \times \mathbb{R}$ via
\begin{equation}\label{Adjoint_bundle_identification_sphere}
 \widetilde{\mathfrak{so}(3)}_a \rightarrow S^2 \times \mathbb{R}, \quad  [\Lambda, \lambda \widehat{\mathbf{e}_z}] \mapsto (\Lambda \mathbf{e}_z, \lambda).
\end{equation}
In particular, it follows that the space of reduced variables $TS^2 \times_Q \widetilde{\mathfrak{so}(3)}_a$ can be identified with $TS^2 \times \mathbb{R}$. The map $\alpha^{(1)}_{\mathcal{A}}$ defined in \eqref{Bundle_diffeomorphism_alpha_1} becomes
\begin{equation}\label{Bundle_diffeomorphism_alpha_1_sphere}
  \alpha^{(1)}_{\mathcal{A}}: TSO(3)/S^1 \rightarrow TS^2 \times \mathbb{R}, \quad [\Lambda, \dot{\Lambda}] \mapsto \big(\mathbf{x}, \dot{\mathbf{x}}, (\Lambda^{-1}\dot{\Lambda})_3\big).
\end{equation}
Moreover, the mapping \eqref{Ad_map_PFB} is, through the correspondence \eqref{Adjoint_bundle_identification_sphere},
\begin{equation}
  S^2 \times \mathbb{R} \rightarrow \mathfrak{so}(3), \quad (\mathbf{x}, \lambda) \mapsto \lambda \widehat{\mathbf{x}}.
\end{equation}
The first-order geodesic equations \eqref{Geodesic_LP_equations} on $SO(3)$ become
\begin{equation}
  \dot{\bar{\mathbf{J}}} = \bar{\boldsymbol{\sigma}} \times \bar{\mathbf{J}}, \quad \dot{\bar{\boldsymbol{\sigma}}} = \bar{\mathbf{J}} \times \bar{\boldsymbol{\sigma}}. \nonumber
\end{equation}
\subsection{Cubics and ballistic curves}\label{Sec-Cubics_and_ballistic_curves}
In this section we show that certain types of ballistic curves in a symmetric space $Q$ are Riemannian cubics. Recall from Remark \ref{Rem_Ballistic_curves} that a ballistic curve is the projection $q(t) \in Q$ of a geodesic $g(t) \in G$. Geodesics in $G$ are given by the Lagrange--Poincar\'e equations \eqref{Geodesic_LP_equations}. As explained in Remark \ref{Remark_LP_symmetric} these are equivalent to
\begin{equation}\label{LP1_equation_rewritten}
  \dot{\bar{\mathbf{J}}} = \left[\bar{\sigma}, \bar{\mathbf{J}}\right], \quad \dot{\bar{\sigma}} = \left[\bar{\mathbf{J}}, \bar{\sigma}\right],
\end{equation}
where $\bar{\mathbf{J}} := \bar{\mathbf{J}}(\dot{q})$.  Recall that for a curve $g(t) \in G$ with projection $q(t) \in Q$ we defined $\sigma = [g, \mathcal{A}(\dot{g})]$ and $\bar{\mathbf{J}} = \bar{\mathbf{J}}(\dot{q})$. As a consequence, the conserved right-invariant velocity is $TR_{g^{-1}} \dot{g} = \bar{\mathbf{J}} + \bar{\sigma}$, where $\bar{\sigma} := i(\sigma)$ with $i$ defined in \eqref{Ad_map_PFB}. Moreover, $\bar{\mathbf{J}}(t) = \operatorname{Ad}_{g(t)} \bar{\mathbf{J}}(0)$ and $\bar{\sigma}(t) =  \operatorname{Ad}_{g(t)}\bar{\sigma}(0)$.  
\begin{theorem}\label{Theorem_ballistic}
  The projection $q(t) \in Q$ of a geodesic $g(t) \in G$ is a Riemannian cubic if and only if at time $t = 0$
  \begin{equation}\label{Iff_ballistic}
    \left[\bar{\sigma}, \left[\bar{\sigma}, \left[\bar{\sigma}, \bar{\mathbf{J}}\right]\right]\right] + \left[\bar{\mathbf{J}}, \left[\bar{\mathbf{J}}, \left[\bar{\mathbf{J}}, \bar{\sigma}\right]\right]\right] = 0.
  \end{equation}
\end{theorem}
\begin{proof}
 Since $\operatorname{Ad}$ is a Lie automorphism it follows from  $\bar{\mathbf{J}}(t) = \operatorname{Ad}_{g(t)} \bar{\mathbf{J}}(0)$ and $\bar{\sigma}(t) =  \operatorname{Ad}_{g(t)}\bar{\sigma}(0)$ that \eqref{Iff_ballistic} holds at $t = 0$ if and only if it holds at all times. Furthermore we obtain from \eqref{LP1_equation_rewritten} that
  \begin{equation}
    \ddot{\bar{\mathbf{J}}} = \left[\bar{\mathbf{J}} + \bar{\sigma}, \dot{\bar{\mathbf{J}}}\right], \quad \dddot{\bar{\mathbf{J}}}  = \left[\bar{\mathbf{J}} + \bar{\sigma}, \ddot{\bar{\mathbf{J}}}\right]. \nonumber
  \end{equation}
Therefore, using also \eqref{Symmetric_space_inclusions_dash} and \eqref{Horizontality_relations_1}, 
\begin{align}
   &\operatorname{H}_q\left(\dddot{\bar{\mathbf{J}}} + 2\left[\ddot{\bar{\mathbf{J}}}, \bar{\mathbf{J}}\right]\right) = \operatorname{H}_q\left(\left[\bar{\sigma} - \bar{\mathbf{J}}, \ddot{\bar{\mathbf{J}}}\right]\right) = \left[\bar{\sigma}, \operatorname{H}_q\left(\ddot{\bar{\mathbf{J}}}\right)\right] - \left[\bar{\mathbf{J}}, \operatorname{V}_q\left(\ddot{\bar{\mathbf{J}}}\right)\right]\nonumber \\
&\qquad = \left[\bar{\sigma}, \left[\bar{\sigma}, \dot{\bar{\mathbf{J}}}\right]\right] - \left[\bar{\mathbf{J}}, \left[\bar{\mathbf{J}}, \dot{\bar{\mathbf{J}}}\right]\right] =  \left[\bar{\sigma}, \left[\bar{\sigma}, \left[\bar{\sigma}, \bar{\mathbf{J}}\right]\right]\right] +  \left[\bar{\mathbf{J}}, \left[\bar{\mathbf{J}}, \left[\bar{\mathbf{J}}, \bar{\sigma}\right]\right]\right].\nonumber
\end{align}
The theorem now follows from equation \eqref{Cubic_symmetric_spaces} for cubics in symmetric spaces.
\end{proof}
It is clear that \eqref{Iff_ballistic} is satisfied if $\bar{\sigma} = 0$ at time $t = 0$. This leads to geodesics $q(t)$ since the projection of a horizontal geodesic is a geodesic. Another class of solutions is given by the following corollaries
\begin{corollary}\label{Corollary_ballistic_example_dash}
  Let $q(t) \in Q$ be the projection of a geodesic $g(t) \in G$ that satisfies $ \left[\bar{\mathbf{J}}, \bar{\sigma}\right] = 0$ at time $t = 0$. Then, $q(t)$ is a Riemannian cubic.
\end{corollary}
\begin{corollary}\label{Corollary_ballistic_example}
  Let $q(t) \in Q$ be the projection of a geodesic $g(t) \in G$ that satisfies, at time $t = 0$,
  \begin{equation}\label{Ballistic_example_relations}
    \left[\bar{\mathbf{J}}, \left[\bar{\mathbf{J}}, \bar{\sigma}\right]\right] = c \bar{\sigma}\,, \quad \left[\bar{\sigma}, \left[\bar{\sigma}, \bar{\mathbf{J}}\right]\right] = c \bar{\mathbf{J}}\,,
  \end{equation}
for $c \in \mathbb{R}$. Then, $q(t)$ is a Riemannian cubic.
\end{corollary}
\begin{remark}\label{Remark_sphere_ballistic}{\rm
For the example $G = SO(3)$, $Q = S^2$, the solutions to equation \eqref{Iff_ballistic} are illustrated in Figure \ref{Figure_ballistic_sphere}. These solutions are fully described by Corollaries \ref{Corollary_ballistic_example_dash} and \ref{Corollary_ballistic_example}. Namely, by means of the identity 
\begin{equation}\notag
  \mathbf{a}\times(\mathbf{b} \times \mathbf{c}) = \mathbf{b}(\mathbf{a} \cdot \mathbf{c}) - \mathbf{c} (\mathbf{a} \cdot \mathbf{b})
\end{equation}
 for vectors in $\mathbb{R}^3$, \eqref{Iff_ballistic} is seen to be equivalent to
\begin{equation}\notag
  \left(\left\|\bar{\boldsymbol{\sigma}}\right\|^2 - \left\|\bar{\mathbf{J}}\right\|^2\right) \bar{\mathbf{J}} \times \bar{\boldsymbol{\sigma}} = 0.
\end{equation}
Possible solutions are given by $\bar{\mathbf{J}} \times \bar{\boldsymbol{\sigma}} = 0$, or by $\left\|\bar{\boldsymbol{\sigma}}\right\|^2 =  \left\|\bar{\mathbf{J}}\right\|^2$ at $t = 0$. The first case is equivalent to $\bar{\mathbf{J}} = 0$ or $\bar{\boldsymbol{\sigma}} = 0$ at $t = 0$, and therefore at all times. This corresponds to trivial projected curves $\mathbf{x}(t) = \mathbf{x}(0) \in S^2$, and to projections of horizontal geodesics, respectively. 

We proceed to analyze the alternative solution given by  $\left\|\bar{\boldsymbol{\sigma}}\right\|^2 =  \left\|\bar{\mathbf{J}}\right\|^2$. 
 For a given initial velocity $\dot{\mathbf{x}}(0) = \mathbf{v} \in T_{\mathbf{x}(0)}S^2$, the projection $\mathbf{x}(t) \in S^2$ describes a constant speed rotation of $\mathbf{x}(0)$ around the axis
\begin{equation}
  \mathbf{\Omega} = \bar{\mathbf{J}} + \bar{\boldsymbol{\sigma}} = \mathbf{x} \times \dot{\mathbf{x}} \pm \left\|\dot{\mathbf{x}}\right\| \mathbf{x}.\nonumber
\end{equation}
Hence, the (constant) rotation axis lies in the plane spanned by $\bar{\mathbf{J}}$ and $\mathbf{x}$, enclosing a $45^\circ$ angle with $\bar{\mathbf{J}}$. The curve $\mathbf{x}(t)$ moves with constant speed along a circle of radius $r = \frac{1}{\sqrt{2}}$. 
}
\end{remark} 
\begin{SCfigure}
\centering
\includegraphics[scale=0.1]{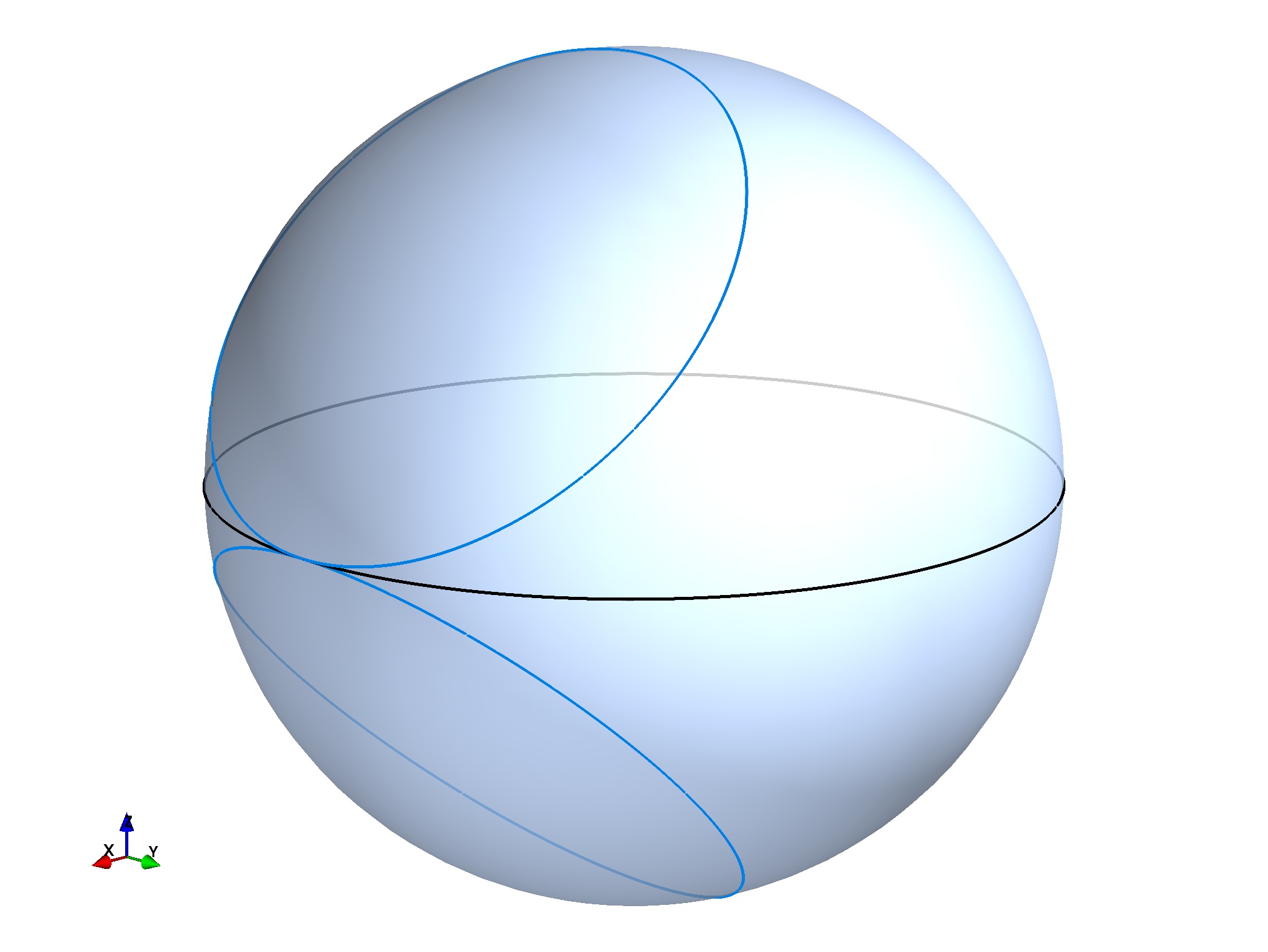}
\caption{\footnotesize Ballistic curves and cubics on the sphere. For given initial position and velocity, there are two types of trajectories that are, simultaneously, projections of geodesics on the rotation group (ballistic curves) \emph{and} Riemannian cubics. The two trajectories are shown for initial position $(1, 0, 0)$ and initial velocity parallel to the $y$-axis. In black a unit speed trajectory along the equator corresponding to the projection of a horizontal geodesic on the rotation group. The blue curves are the circular unit-speed trajectories of radius $\frac{1}{\sqrt{2}}$  described in Corollary \ref{Corollary_ballistic_example} and Remark \ref{Remark_sphere_ballistic}. }
\label{Figure_ballistic_sphere}
\end{SCfigure}
\subsection{Second-order Lagrange--Poincar\'e reduction}
To discuss second-order reduction, we introduce the quotient $T^{(2)}G/G_a$ of $T^{(2)}G$ by the natural action of $G_a$, whose elements will be denoted by $[g, \dot{g}, \ddot{g}] \in T^{(2)}G/G_a$. We will make use of the bundle diffeomorphism $\alpha^{(2)}_{\mathcal{A}}: T^{(2)}G/G_a \rightarrow T^{(2)}Q \times_Q 2\widetilde{\mathfrak{g}}_a\,,$
\begin{equation}\label{Bundle_diffeomorphism_alpha_2}
 \left[g, \dot{g}, \ddot{g}\right] \mapsto  (q, \dot{q}, \ddot{q}) \times \left[g, \mathcal{A}(\dot{g})\right] \oplus \frac{D^{\mathcal{A}}}{Dt} \left[g, \mathcal{A}(\dot{g})\right]    = (q, \dot{q}, \ddot{q}) \times \left[g, \mathcal{A}(\dot{g})\right] \oplus \left[g, \partial_t\mathcal{A}(\dot{g})\right],
\end{equation}
see \cite{Gay-BalmazEtAl2011HOLPHP}.
Here, we defined $q(t) := \Pi\left(  g(t)\right) $, where $g(t)$ is a curve representing $[g, \dot{g}, \ddot{g}]$, that is, 
$(q, \dot{q}, \ddot{q}) = T^{(2)}_g\Pi(g, \dot{g}, \ddot{g})$. Denoting the right-invariant velocity by $\xi = TR_{g^{-1}} \dot{g} $ and using the definition \eqref{PFB_mechanical_connection} of 
$\mathcal{A}$ together with \eqref{Ad_relations_projections}, 
this becomes
\begin{equation}\label{alpha_2_bundle_map}
  [g, \dot{g}, \ddot{g}] \mapsto (q, \dot{q}, \ddot{q}) \times [g, \operatorname{Ad}_{g^{-1}}\operatorname{V}_q(\xi)] \oplus [g,\operatorname{Ad}_{g^{-1}} \operatorname{V}_q(\dot{\xi})].
\end{equation}
\begin{remark}{\rm
In the example $G = SO(3)$, $Q = S^2$, working with the conventions laid out in  Remark \ref{Remark_Conventions}, the space of reduced variables $T^{(2)}S^2 \times_Q 2\widetilde{\mathfrak{so}(3)}_a$ can be identified with $T^{(2)}S^2 \times \mathbb{R}^2$. The map $\alpha^{(2)}_{\mathcal{A}}$  is
\begin{equation}\label{Bundle_diffeomorphism_alpha_2_sphere}
  \alpha^{(2)}_{\mathcal{A}}: T^{(2)}SO(3)/S^1 \rightarrow T^{(2)}S^2 \times \mathbb{R}^2, \quad [\Lambda, \dot{\Lambda}, \ddot{\Lambda}] \mapsto \big(\mathbf{x}, \dot{\mathbf{x}}, \ddot{\mathbf{x}}, (\Lambda^{-1}\dot{\Lambda})_3, \partial_t(\Lambda^{-1}\dot{\Lambda})_3\big).
\end{equation}
}
\end{remark} 
\paragraph{The reduced Lagrangian.} For $L: T^{(2)}G \rightarrow \mathbb{R}$ a $G_a$-invariant Lagrangian we define the reduced Lagrangian $\ell$ by $L = \ell \circ \alpha^{(2)}_{\mathcal{A}}$.  Consider the Lagrangian for Riemannian cubics $L = \frac{1}{2} \left\|\frac{D}{Dt} \dot{g}\right\|^2_g$ and note the following equalities,
\begin{align}
  L(g, \dot{g}, \ddot{g}) &= \frac{1}{2} \left\|\frac{D}{Dt}\dot{g}\right\|_g^2 = \frac{1}{2} \left\|\dot{\xi}\right\|_{\mathfrak{g}}^2 
= \frac{1}{2} \left\|\operatorname{H}_q\big(\dot{\xi}\big)\right\|_{\mathfrak{g}}^2 + \frac{1}{2} \left\|\operatorname{V}_q\big(\dot{\xi}\big)\right\|_{\mathfrak{g}}^2\nonumber \\ &= \frac{1}{2} \left\|\big(\dot{\xi}\big)_Q (q)\right\|_q^2 + \frac{1}{2} \left\|\operatorname{V}_q\big(\dot{\xi}\big)\right\|_{\mathfrak{g}}^2 \nonumber \\ &= \frac{1}{2} \left\|\frac{D}{Dt}\dot{q}  - \nabla_{\dot{q}} \left(\operatorname{V}_q (\xi)\right)_Q\right\|_q^2 + \frac{1}{2} \left\|\operatorname{V}_q\big(\dot{\xi}\big)\right\|_{\mathfrak{g}}^2, \label{Lagrangian_rewritten_PFB}
\end{align}
where we used the right-invariance of $L$, the definition of the normal metric, and part ${\bf (iii)}$ of Proposition \ref{Covariant_derivative_main_proposition}.  It follows from 
\eqref{alpha_2_bundle_map} and \eqref{Lagrangian_rewritten_PFB} that the reduced Lagrangian $\ell: T^{(2)}Q \times_Q 2 \tilde{\mathfrak{g}}_a \rightarrow \mathbb{R}$ reads
\begin{equation}\label{Reduced_spline_Lagrangian_1}
\ell (q, \dot{q}, \ddot{q}, \sigma , \dot{\sigma} ) = \frac{1}{2} \left\|\frac{D}{Dt} \dot{q} - \nabla_{\dot{q}}\bar{\sigma}_Q\right\|^2_q + \frac{1}{2} \|\dot{\sigma} \|^2_{\tilde{\mathfrak{g}}_a},
\end{equation}
where we recall $\dot{\sigma} := \frac{D^{\mathcal{A}}}{Dt} \sigma$. The reduced Lagrangian therefore measures the deviations from the geodesic Lagrange--Poincar\'e equations \eqref{Geodesic_LP_equations}.
\begin{remark}{\rm
In the example $G = SO(3)$, $Q = S^2$, the reduced Lagrangian $\ell: T^{(2)}S^2 \times \mathbb{R}^2 \rightarrow \mathbb{R}$ is
\begin{equation}
  (\mathbf{x}, \dot{\mathbf{x}}, \ddot{\mathbf{x}}, \sigma, \dot{\sigma}) \mapsto \frac{1}{2} \left\|D_t\dot{\mathbf{x}} - \sigma \mathbf{x} \times \dot{\mathbf{x}}\right\|_{\mathbf{x}}^2 + \frac{1}{2} \dot{\sigma}^2 = \frac{1}{2}\left\|D_t\dot{\mathbf{x}}\right\|_{\mathbf{x}}^2 - \sigma D_t\dot{\mathbf{x}} \cdot (\mathbf{x} \times \dot{\mathbf{x}}) + \frac{1}{2} \sigma^2\left\|\dot{\mathbf{x}}\right\|_{\mathbf{x}}^2 + \frac{1}{2}\dot{\sigma}^2, \nonumber
\end{equation}
where the norm $\|\cdot \|_{\mathbf{x}}$ is evaluated as the standard Euclidean norm. 
}
\end{remark} 
\paragraph{Coupling.}
The reduced Lagrangian couples the horizontal and vertical parts of the motion through the term $\nabla_{\dot{q}}\bar{\sigma}_Q$. This explains the absence of a general horizontal lifting property for Riemannian cubics studied in Section \ref{Sec_Riemannian_submersion_viewpoint}. Indeed, let us instead define the Lagrangian
\begin{equation}
  L_{KK}: T^{(2)}G \rightarrow \mathbb{R}, \quad (g, \dot{g}, \ddot{g}) \mapsto \frac{1}{2} \left\|\frac{D}{Dt}T\Pi(\dot{g})\right\|^2_{\Pi(g)} + \frac{1}{2} \left\|\partial_t\mathcal{A}(\dot{g})\right\|^2_{\mathfrak{g}}
\end{equation}
with reduced Lagrangian 
\begin{equation}\label{L_KK_reduced}
  \ell_{KK}: T^{(2)}Q \times_Q 2 \tilde{\mathfrak{g}}_a \rightarrow \mathbb{R}, \quad   (q, \dot{q}, \ddot{q}, \sigma, \dot{\sigma})  \mapsto \frac{1}{2} \left\|\frac{D}{Dt}\dot{q}\right\|^2_q + \frac{1}{2} \left\|\dot{\sigma}\right\|^2_{\tilde{\mathfrak{g}}_a}.
\end{equation}
The Lagrangian $L_{KK}$ belongs to a class of Lagrangians that were studied in \cite{Gay-BalmazEtAl2011HOLPHP} as natural second-order generalizations of the Kaluza-Klein Lagrangian. The reduced Lagrangians $\ell$ in \eqref{Reduced_spline_Lagrangian_1} and $\ell_{KK}$ in \eqref{L_KK_reduced} differ by the coupling term $\nabla_{\dot{q}}\bar{\sigma}_Q$. The decoupled form of $\ell_{KK}$ leads to a general horizontal lifting theorem. Namely, any horizontal lift $g(t)$ to $G$ of a cubic spline $q(t)$ on $Q$ is a critical point of the action $\int L_{KK} \, dt$.
\paragraph{Lagrange--Poincar\'e equations.} We now compute the Lagrange--Poincar\'e equations. Taking $ \varepsilon $-variations and defining $V_q:= \frac{D}{Dt} \dot{q} - \nabla _{\dot q} \bar\sigma_Q$, we have, for the first term of \eqref{Reduced_spline_Lagrangian_1},
\[
\delta \int_0^1\frac{1}{2} \left\|\frac{D}{Dt} \dot{q} - \nabla _{\dot q} \bar\sigma_Q\right\|^2_q dt=\int_0^1 \gamma _Q \left( V_q,\frac{D}{D\varepsilon}\frac{D}{Dt} \dot{q}- \frac{D}{D\varepsilon} \nabla _{\dot q} \bar\sigma_Q\right) dt.
\]
We then compute
\begin{align*}
\frac{D}{D\varepsilon}\frac{D}{Dt} \dot{q}= \frac{D}{Dt}\frac{D}{D\varepsilon } \dot{q}+R( \delta q, \dot q) \dot q
\end{align*} 
and
\begin{align*} 
\frac{D}{D\varepsilon} \nabla _{\dot q} (\bar\sigma _1)_Q&= \frac{D}{D\varepsilon} \left( \frac{D}{Dt} \bar\sigma_Q(q) - (\partial _t \bar \sigma) _Q(q) \right) \\
&= \frac{D}{Dt} \frac{D}{D\varepsilon} \bar\sigma_Q(q)+R\left(  \delta q, \dot q\right) \bar\sigma_Q(q) - \frac{D}{D\varepsilon}  (\partial _t \bar \sigma) _Q(q)\\
&=\frac{D}{Dt} \left( \nabla _{ \delta q}  \bar\sigma_Q+ \left( \delta \bar \sigma \right) _Q (q) \right) +R\left(  \delta q, \dot q\right) \bar\sigma_Q(q) - \left( \delta \partial _t \bar \sigma \right) _Q(q) - \nabla _{ \delta q} \left( \partial _t \bar\sigma\right) _Q
\end{align*}
Lemma \ref{Lemma_delta_sigma_bar} shows that
\[
\left( \delta \bar \sigma \right) _Q (q)= \left(i(\delta \sigma) \right) _Q (q)+ \left( F( \delta q, \sigma ) \right) _Q(q)=\left( F( \delta q, \sigma ) \right) _Q(q),
\]
since $\delta \sigma:= \left.\frac{D}{D\varepsilon} \right|_{ \varepsilon =0}\sigma _\varepsilon \in \mathfrak{g}  ^V _q$.  So we have
\begin{align*} 
&\int_0^1\gamma _Q \left(V _q , \frac{D}{Dt} \left( \delta \bar \sigma \right) _Q(q) \right) dt = -\int_0^1 \gamma _Q \left( \frac{D}{Dt} V_q,\left( F( \delta q, \sigma ) \right) _Q(q)\right)  dt\\
&\quad = -\int_0^1 \gamma  \left( \bar{\mathbf{J}} \left( \frac{D}{Dt} V_q \right) , F( \delta q, \sigma ) \right) dt=-\int_0^1 \gamma_Q   \left(F^T_{ \sigma } \bar{\mathbf{J}} \left( \frac{D}{Dt} V_q \right) ,\delta q\right) dt
\end{align*} 
and
\begin{align*} 
&\int_0^1 \gamma _Q \left(V _q , ( \delta \partial _t \bar\sigma )_Q (q) \right) dt= \int_0^1 \gamma \left( \bar{ \mathbf{J} }(V_q), \partial _t \delta \bar \sigma \right) =-\int_0^1 \gamma \left(\partial _t  \bar{ \mathbf{J} }(V_q), \delta \bar \sigma \right)\\
& \quad =-\int_0^1 \gamma \left(\partial _t  \bar{ \mathbf{J} }(V_q), i(\delta \sigma )+F( \delta q, \sigma )\right)=-\int_0^1 \bar{\gamma} \left(i^T_q \partial _t  \bar{ \mathbf{J} }(V_q), \delta \sigma \right) +\gamma _Q \left( F_\sigma^T \partial _t  \bar{ \mathbf{J} }(V_q), \delta q\right),
\end{align*} 
where $F_{ \sigma _q}^T: \mathfrak{g}  \rightarrow T_qQ$ is the transpose of the map $F_ {\sigma_q} :T_qQ \rightarrow \mathfrak{g}  $, $F_ {\sigma _q}(v_q):= F(v _q , \sigma _q )$, and $i_q^T: \mathfrak{g}  \rightarrow(\tilde{ \mathfrak{g}  }_a)_q$ is the transpose of the map $i_q:(\tilde{ \mathfrak{g}  }_a)_q \rightarrow \mathfrak{g}  $  (the restriction of \eqref{Ad_map_PFB} to the fiber $(\tilde{ \mathfrak{g}  }_a)_q$ of $\tilde{ \mathfrak{g}  }_a$ at $q$). 

For the second term of \eqref{Reduced_spline_Lagrangian_1}, we have
\[
\delta \int_0^1\frac{1}{2} \|\dot{\sigma} \|^2_{\tilde{\mathfrak{g}}_a}dt= \int_0^1 \bar \gamma \left( \dot{\sigma} , \frac{D^ \mathcal{A} }{D\varepsilon} \dot{\sigma} \right) dt.
\]

Using the variations
\[
\delta \sigma =\frac{D^ \mathcal{A} }{Dt} \eta -[ \sigma , \eta ]+\widetilde{\mathcal{B}}( \delta q, \dot q)\in \tilde{ \mathfrak{g}  }_a
,\quad 
\delta \dot{\sigma} = \frac{D^{\mathcal{A}}}{Dt} \delta \sigma - [ \widetilde{ \mathcal{B} }(\dot q , \delta q), \sigma ]\in \tilde{ \mathfrak{g}  }_a,
\]
and the formula
\[
\frac{d}{dt} \mathbf{J} \left( \alpha (t) \right) - \mathbf{J} \left( \frac{D}{Dt} \alpha (t) \right) = \mathcal{F} ^ \nabla ( \alpha (t) , \dot q (t) ),
\]
where $\left\langle  \mathcal{F} ^ \nabla ( \alpha _q , v _q ), \eta \right\rangle := \left\langle \alpha _q , \nabla _{ v _q } \eta _Q \right\rangle $, (see \cite{GBHoRa2010})  we get the equations 
\begin{align*} 
&\frac{D^2}{Dt^2}V_q+ \nabla \bar \sigma _Q^T \cdot \frac{D}{Dt} V_q+ \nabla ( \partial _t \bar{ \sigma })_Q^T \cdot V _q +R( V_q , \dot q- \bar{ \sigma }_Q(q)) \dot q \\
&\qquad  + \left\langle \frac{D^ \mathcal{A} }{Dt} \dot{\sigma} + \operatorname{ad}^\dagger_ \sigma \dot{\sigma} +i_q^T \partial _t {\bar{\mathbf{J}}} (V_q), \mathbf{i} _{\dot q} \widetilde{ \mathcal{B} } \right\rangle ^\sharp = F_ \sigma ^T {\left(\mathcal{F} ^ \nabla \left( V_q^\flat, \dot q \right)\right)^\sharp} \\
&\left(  \frac{D^ \mathcal{A} }{Dt} + \operatorname{ad}^\dagger_ \sigma \right)  \left( \frac{D^ \mathcal{A} }{Dt} \dot{\sigma}  +i_q^T \partial _t \bar{\mathbf{J}} (V_q)\right) =0,
\end{align*} 
where we recall that $V_q:= \frac{D}{Dt} \dot{q} - \nabla _{\dot q} \bar\sigma_Q\in TQ$.

Using \eqref{B_tilde} these  equations can be rewritten as
\begin{align}
  &\frac{D^3}{Dt^3}\dot{q} + R\left(\frac{D}{Dt}\dot{q}, \dot{q}\right) \dot{q} = \frac{D^2}{Dt^2} \nabla_{\dot{q}}\bar{\sigma}_Q -  \nabla \bar \sigma _Q^T \cdot \frac{D}{Dt} V_q -  \nabla ( \partial _t \bar{ \sigma })_Q^T \cdot V _q + R\left(\frac{D}{Dt} \dot{q}, \bar{\sigma}_Q(q)\right) \dot{q} \nonumber\\
&\qquad + R(\nabla_{\dot{q}}\bar{\sigma}_Q, \dot{q} - \bar{\sigma}_Q(q)) \dot{q}
+ \nabla_{\dot{q}} \left(i\left(\ddot{\sigma} + \operatorname{ad}^\dagger_\sigma \dot{\sigma} + i_q^T \partial _t \bar{\mathbf{J}} (V_q)\right)\right)_Q +  F_ \sigma ^T \left(\mathcal{F} ^ \nabla \left( V_q^\flat, \dot q \right)\right)^\sharp \label{LP_final_1} \\
&\left(  \frac{D^ \mathcal{A} }{Dt} + \operatorname{ad}^\dagger_ \sigma \right)  \left(\ddot{\sigma}  +i_q^T \partial _t \bar{\mathbf{J}} (V_q)\right) =0. \label{LP_final_2}
\end{align}
These equations are the second-order analogue of \eqref{Geodesic_LP_equations}. The left hand side of the first one is the equation for Riemannian cubics on $Q$. Hence the right hand side of \eqref{LP_final_1} is the obstruction for the projected curve to be a cubic. For symmetric spaces, solutions to \eqref{LP_final_1}-\eqref{LP_final_2} with vanishing obstruction include in particular the curves characterized in Theorem \ref{Lifting_symmetric_spaces}, but also the special geodesics on the group that project to the ballistic curves of Section \ref{Sec-Cubics_and_ballistic_curves}.  
\begin{remark}{\rm
For $G = SO(3)$, $Q = S^2$, the Lagrange--Poincar\'e equations are computed as
\begin{align}\label{SO3_S2} 
     &\frac{D^3}{Dt^3}\dot{\mathbf{x}} + R\left(\frac{D}{Dt}\dot{\mathbf{x}}, \dot{\mathbf{x}}\right) \dot{\mathbf{x}}= \sigma^2\ddot{\mathbf{x}}^\perp + 2\sigma \dot{\sigma}\dot{\mathbf{x}} + \ddot{\sigma}(\mathbf{x}\times \dot{\mathbf{x}}) + 3\dot{\sigma}(\mathbf{x}\times \ddot{\mathbf{x}})\nonumber  \\ &\qquad \qquad +\sigma\big(2[  \mathbf{x}\times \dddot{\mathbf{x}} + (\dot{\mathbf{x}}\times \ddot{\mathbf{x}})^\perp] + \|\dot{\mathbf{x}}\|^2(\mathbf{x} \times \dot{\mathbf{x}})\big) - \alpha(\mathbf{x}\times \dot{\mathbf{x}}) \nonumber \\
&-\mathbf{x} \cdot (\dot{\mathbf{x}} \times \ddot{\mathbf{x}}) + \sigma\|\dot{\mathbf{x}}\|^2 - \ddot{\sigma} = \alpha\,, \nonumber
\end{align}
for a constant $\alpha \in \mathbb{R}$. Here we denoted by $\mathbf{v}^\perp = \mathbf{v} - \mathbf{x}(\mathbf{x}\cdot \mathbf{v})$ the orthogonal projection of $\mathbf{v}$ onto the tangent plane to $S^2$ at $\mathbf{x}$.
}
\end{remark}
\section{Summary and Outlook}
This paper has investigated Riemannian cubics in object manifolds with normal metrics and in particular their relation to Riemannian cubics on the Lie group of transformations.

Our starting point was the definition, in Section \ref{Section-Geometric_Setting}, of necessary concepts and a treatment of covariant derivatives for normal metrics in Section \ref{Section-Covariant_derivatives}. The derivation of the Euler-Lagrange equation for cubics from the viewpoint of normal metrics followed in Section \ref{Section-Cubic_splines_for_normal_metrics}. The examples of Lie groups with invariant metrics and symmetric spaces were discussed in detail, and the relation with equations previously present in the literature was clarified. The new form of the equation was seen to lend itself to the analysis of horizontal lifts of cubics, due to the appearance of the horizontal generator of curves.

Section \ref{Sec-Horizontal_lifts} proceeded with this line of investigation by deriving several results about horizontal lifting properties of cubics. For symmetric spaces a complete characterization was achieved of the cubics that can be lifted horizontally to cubics on the group of isometries. In rank-one symmetric spaces this selects geodesics composed with cubic polynomials in time. The section continued with a treatment of the corresponding question in the context of Riemannian submersions. 
In Section \ref{Section-Extended_analysis_Reduction_by_isotropy_subgroup} certain non-horizontal geodesics on the group were shown to project to cubics in the object manifold. A complete characterization of such geodesics was given in the sense of Theorem \ref{Theorem_ballistic}. For the unit sphere acted on by the rotation group the corresponding projections were seen to be the circles of radius ${1}/{\sqrt{2}}$. A discussion of Lagrange--Poincar\'e reduction of cubics led to reduced equations that identified the obstruction for projections of cubics to be cubics in the object manifold.

Future research should seek to extend the results found here for horizontal lifts of cubics in symmetric spaces to more general situations. A first step in this direction has been taken in the context of Riemannian submersions in Section \ref{Sec_Riemannian_submersion_viewpoint}. However, a more complete characterization of the cubics that lift horizontally to cubics may provide a wider class of applications. A similar remark holds for the analysis of the special ballistic curves that satisfy the equation for cubics. All of these problems are tightly linked to the obstruction term in the Lagrange--Poincar\'e equations, particularly on the right hand side of \eqref{LP_final_1}. Hence, one of the main tasks ahead is to deepen the understanding of that obstruction term, and thereby determine additional situations in which it vanishes.
\subsection*{Acknowledgements}
We thank B. Doyon, P. Michor, L. Noakes and  A. Trouv\'e for encouraging comments and insightful remarks during the course of this work. DDH, DMM and FXV are grateful for partial support by the Royal Society of London Wolfson Research Merit Award  and the European Research Council Advanced Grant. FGB has been partially supported by a ``Projet Incitatif de Recherche'' contract from the Ecole Normale Superieure de Paris. TSR was partially supported by Swiss NSF grant 200020-126630 and by the government grant of the Russian Federation for support of research projects implemented by leading scientists, Lomonosov Moscow State University under the agreement No. 11.G34.31.0054. 
\small{
\bibliographystyle{alpha}
\bibliography{myrefs-FGBDM}}

\end{document}